\documentclass[10pt]{amsart}

\usepackage{amssymb}
\usepackage[leqno]{amsmath}
\usepackage{mathrsfs}
\usepackage{stmaryrd}
\usepackage{chemarrow}
\usepackage{diagbox}
\usepackage{enumerate}
\usepackage{graphicx}
\usepackage[pdftex,bookmarksnumbered,bookmarksopen,colorlinks,linkcolor=red,anchorcolor=black,citecolor=blue,urlcolor=blue]{hyperref}
\usepackage[all]{xy}
\usepackage{tikz}
\usetikzlibrary{arrows}
\usepackage{booktabs}
\usepackage{multirow}

\newcommand{\normmm}[1]{{\left\vert\kern-0.25ex\left\vert\kern-0.25ex\left\vert #1 
   \right\vert\kern-0.25ex\right\vert\kern-0.25ex\right\vert}}

  \newcounter{mnote}
  \let\oldmarginpar\marginpar
    \renewcommand\marginpar[1]{\-\oldmarginpar[\raggedleft\footnotesize #1]%
    {\raggedright\footnotesize #1}}

\newtheorem{theorem}{Theorem}[section]
\newtheorem{lemma}[theorem]{Lemma}

\newtheorem{example}[theorem]{Example}

\newcommand{\dd}{\,{\rm d}}

\newcommand{\curl}{\operatorname{curl}}
\renewcommand{\div}{\operatorname{div}}

\numberwithin{equation}{section}

\begin{document}
\title[Morley-Wang-Xu Element Method]{A Morley-Wang-Xu element method for a fourth order elliptic singular perturbation problem}

\author{Xuehai Huang}%
\address{School of Mathematics, Shanghai University of Finance and Economics, Shanghai 200433, China}%
\email{huang.xuehai@sufe.edu.cn}%
\author{Yuling Shi}%
\address{School of Mathematics, Shanghai University of Finance and Economics, Shanghai 200433, China}%
\email{shiyuling@163.sufe.edu.cn}%
\author{Wenqing Wang}%
\address{Department of Basic Teaching, Wenzhou Business College, Wenzhou 325035, China}%
\email{wangwenqing81@hotmail.com}%

\thanks{This work was supported by the National Natural Science Foundation of China Project (Grant Nos. 11771338 and 12071289), and the Fundamental Research Funds for the Central Universities (Grant No. 2019110066).}

\makeatletter
\@namedef{subjclassname@2020}{\textup{2020} Mathematics Subject Classification}
\makeatother

\subjclass[2020]{
  65N12;   
  65N22;   
  65N30;   
  65F08;   
}
\keywords{Fourth order elliptic singular perturbation problem, Morley-Wang-Xu element, Decoupling, Boundary layers, Fast solver}

\begin{abstract}
  A Morley-Wang-Xu (MWX) element method with a simply modified right hand side is proposed for a fourth order elliptic singular perturbation problem, in which the discrete bilinear form is standard as usual nonconforming finite element methods. The sharp error analysis is given for this MWX element method. And the Nitsche's technique is applied to the MXW element method to achieve the optimal convergence rate in the case of the boundary layers.
  An important feature of the MWX element method is solver-friendly. Based on a discrete Stokes complex in two dimensions, the MWX element method is decoupled into one Lagrange element method of Poisson equation,
  two Morley element methods of Poisson equation and one nonconforming $P_1$-$P_0$ element method of Brinkman problem, which implies efficient and robust solvers for the MWX element method.
  Some numerical examples are provided to verify the theoretical results.
\end{abstract}

\maketitle


\section{Introduction}


In this paper, we shall apply the Morley-Wang-Xu (MWX) element \cite{Morley1968, WangXu2006} to discretize the fourth order elliptic singular perturbation problem
\begin{equation}\label{fourthorderpertub}
  \begin{cases}
    \varepsilon^2\Delta^{2}u-\Delta u=f \quad\;\; \textrm{in}~\Omega, \\
    u=\partial_nu=0 \quad\quad\quad \textrm{on}~\partial\Omega,
  \end{cases}
\end{equation}
where $\Omega\subset \mathbb{R}^d$ with $d\geq2$ is a convex and bounded
polytope,  $f\in L^2(\Omega)$, $n$ is the unit outward normal to $\partial\Omega$, and $\varepsilon$ is a real small and positive parameter.

The MWX element is the simplest finite element for fourth order problems, as it has the fewest number of degrees of freedom on each element. The generalization of the  MWX element to any star-shaped polytope in any dimension is recently developed in the context of the virtual element in \cite{ChenHuang2020}.
However it is divergent to discretize problem \eqref{fourthorderpertub} by the MWX element in the following standard way when $\varepsilon$ is very close to $0$ \cite{NilssenTaiWinther2001, Wang2001}: find $u_{h0} \in V_{h0} $ such that
\begin{equation}\label{mmwx0intro0}
  \varepsilon^2 (\nabla_h^{2}u_{h0}, \nabla_h^{2}v_h)+(\nabla_h u_{h0}, \nabla_h v_h)=(f, v_h) \quad \forall~v_h \in V_{h0},
\end{equation}
where $V_{h0}$ is the global MWX element space. To this end, a modified MWX element method was advanced in \cite{WangXuHu2006, WangMeng2007} to deal with this divergence by replacing $(\nabla_h u_{h0}, \nabla_h v_h)$ with $(\nabla\Pi_h u_{h0}, \nabla\Pi_h v_h)$, where $\Pi_h$ is an interpolation operator from $V_{h0}$ to some lower-order $H^1$-conforming finite element space.
Instead of introducing the interpolation operator, the combination of the MWX element and the interior penalty discontinuous Galerkin formulation \cite{Arnold1982} is proposed to discretize problem \eqref{fourthorderpertub} in \cite{WangHuangTangZhou2018}.
Both modified Morley element methods in \cite{WangXuHu2006, WangMeng2007, WangHuangTangZhou2018} uniformly converge with respect to the parameter~$\varepsilon$.

Apart from the MWX element, there are many other $H^2$-nonconforming elements constructed to design robust numerical methods for problem \eqref{fourthorderpertub}, including $C^0$ $H^2$-nonconforming elements in \cite{NilssenTaiWinther2001, TaiWinther2006, GuzmanLeykekhmanNeilan2012, WangWuXie2013, ChenChenQiao2013, ChenChen2014, XieShiLi2010, ChenChenXiao2014, WangShiXu2007, WangShiXu2007a} and fully $H^2$-nonconforming elements in \cite{ChenZhaoShi2005, ChenLiuQiao2010, WangWuXie2013}.
And a $C^0$ interior penalty discontinuous Galerkin (IPDG) method with the Lagrange element space was devised for problem~\eqref{fourthorderpertub} in \cite{BrennerNeilan2011,FranzRoosWachtel2014}.
We refer to \cite{Semper1992, ArgyrisFriedScharpf1968, Zhang2009a} for the $H^2$-conforming finite element methods of problem \eqref{fourthorderpertub}, which usually suffer from large number of degrees of freedom.

To design a simple finite element method for problem \eqref{fourthorderpertub}, we still employ the MWX element space and the standard discrete bilinear formulation as the left hand side of the discrete method \eqref{mmwx0intro0} in this paper. We simply replace the right hand side $(f, v_h)$ by $(f, P_hv_h)$, where $P_h$ is the $H^1$-orthogonal projector onto the $H^1$-conforming $\ell$th order Lagrange element space $W_h$ with $\ell=1,2$. In a word, we propose the following robust MWX element method for problem~\eqref{fourthorderpertub}:
find $u_{h0} \in V_{h0} $ such that
\begin{equation}\label{mmwx0intro}
  \varepsilon^2 (\nabla_h^{2}u_{h0}, \nabla_h^{2}v_h)+(\nabla_h u_{h0}, \nabla_h v_h)=(f, P_hv_h) \quad \forall~v_h \in V_{h0}.
\end{equation}
The stiffness matrix of the discrete method~\eqref{mmwx0intro} can be assembled in a standard way, which is sparser than that of the discontinuous Galerkin methods, such as \cite{WangHuangTangZhou2018}.
After establishing the interpolation error estimate and consistency error estimate, the optimal convergence rate $O(h)$ of the energy error is achieved. And the discrete method~\eqref{mmwx0intro} possesses the sharp and uniform convergence rate $O(h^{1/2})$ of the energy error in consideration of the boundary layers.

An important feature of the discrete method~\eqref{mmwx0intro} is solver-friendly. First the discrete method~\eqref{mmwx0intro} is equivalent to find
$w_h\in W_h$ and $u_{h0}\in V_{h0}$ such that
\begin{align}
  (\nabla w_h, \nabla \chi_h)                   & =(f, \chi_h)  \quad\quad\quad\quad\; \forall~\chi_h\in W_h, \label{intro:mmwx0equiv1}  \\
  \varepsilon^2 a_h(u_{h0},v_h)+b_h(u_{h0},v_h) & =(\nabla w_h, \nabla_h v_h) \quad\;\; \forall~v_h\in V_{h0}. \label{intro:mmwx0equiv2}
\end{align}
Especially in two dimensions, thanks to the relationship between the Morley element space $V_{h0}$ and the vectorial nonconforming $P_1$ element space $V_{h0}^{CR}$ (cf. \cite[Theorem 4.1]{FalkMorley1990}),
the discrete method \eqref{intro:mmwx0equiv2} can be decoupled into
two Morley element methods of Poisson equation and one nonconforming $P_1$-$P_0$ element method of Brinkman problem, i.e.,
find $(z_h, \phi_h, p_h, w_h)\in V_{h0}\times V_{h0}^{CR}\times \mathcal Q_{h}\times V_{h0}$ such that
\begin{subequations}
  \begin{align}
    (\curl_h z_h, \curl_h v_h)                                                           & =(\nabla w_h, \nabla_h v_h)\quad\, \forall~v_h\in V_{h0},  \label{intro:4thpurbmfem0a}          \\
    (\phi_h, \psi_h)+\varepsilon^2(\nabla_h\phi_h, \nabla_h\psi_h) + (\div_h\psi_h, p_h) & =(\curl_h z_h, \psi_h)
    \quad \forall~\psi_h\in  V_{h0}^{CR}, \label{intro:4thpurbmfem0b}                                                                                                                      \\
    (\div_h\phi_h, q_h)                                                                  & = 0\quad\quad\quad\quad\quad\quad\, \forall~q_h\in \mathcal Q_{h},  \label{intro:4thpurbmfem0c} \\
    (\curl_h u_{h0}, \curl_h \chi_h)                                                     & = (\phi_h, \curl_h\chi_h) \quad \forall~\chi_h\in V_{h0}. \label{intro:4thpurbmfem0d}
  \end{align}
\end{subequations}
When $\varepsilon$ is small, the discrete method~\eqref{intro:mmwx0equiv2} can be easily solved by the the conjugate gradient (CG) method with the auxiliary space preconditioner \cite{Xu1996}.
The decoupling \eqref{intro:4thpurbmfem0a}-\eqref{intro:4thpurbmfem0d} will induce efficient and robust solvers for the MWX element method~\eqref{mmwx0intro} for large $\varepsilon$.
The Lagrange element method of Poisson equation~\eqref{intro:mmwx0equiv1}, and
the Morley element methods of Poisson equation~\eqref{intro:4thpurbmfem0a} and \eqref{intro:4thpurbmfem0d} can be solved by the CG method with the auxiliary space preconditioner,
in which the $H^1$ conforming linear element discretization on the same mesh for the
Poisson equation can be adopted as the auxiliary problem.
And the algebraic multigrid (AMG) method is used to solve the auxiliary problem.
As for the nonconforming $P_1$-$P_0$ element methods of Brinkman problem~\eqref{intro:4thpurbmfem0b}-\eqref{intro:4thpurbmfem0c}, we can use the block-diagonal preconditioner in \cite{OlshanskiiPetersReusken2006,MardalWinther2004,CahouetChabard1988} or the approximate block-factorization preconditioner in \cite{ChenHuHuang2018}, which are robust with respect to the mesh size $h$.
The resulting fast solver of the MWX element method~\eqref{mmwx0intro} also works for the shape-regular unstructured meshes.

When $\varepsilon$ is close to zero, the uniform convergence rate $O(h^{1/2})$ of the energy error of the discrete method~\eqref{mmwx0intro}
is sharp but not optimal, where the optimal convergence rate should be $O(h^{\ell})$ for $\ell=1,2$.
To promote the convergence rate in the case of the boundary layers,  we apply the Nitsche's technique in \cite{GuzmanLeykekhmanNeilan2012} to the discrete method~\eqref{mmwx0intro}, i.e. impose the boundary condition $\partial_nu=0$ weakly by the penalty technique \cite{Arnold1982}.
The optimal error analysis is present for the resulting discrete method, whose convergence rate is uniform with respect to the perturbation parameter $\varepsilon$ when $\varepsilon$ approaches zero.
Similarly, as \eqref{intro:4thpurbmfem0a}-\eqref{intro:4thpurbmfem0d}, the discrete method with Nitsche's technique on the boundary can also be decoupled into one Lagrange element method of Poisson equation,
two Morley element methods of Poisson equation and one nonconforming $P_1$-$P_0$ element method with Nitsche's technique of Brinkman problem, which is solver-friendly.

The rest of this paper is organized as follows. Some notations, connection operators and interpolation operators are shown in section 2. In section 3, we develop and analyze the MWX element method, and the MWX element method with Nitsche's technique is devised and analyzed in section 4. Section~5 focuses on the equivalent decoupling of the MWX element methods. Finally, some numerical results are given in section 6  to confirm the theoretical results.


\section{Connection Operators and Interpolation Operators}
\subsection{Notation}
Given a bounded domain $G\subset\mathbb{R}^{d}$ and a
non-negative integer $m$, let $H^m(G)$ be the usual Sobolev space of functions
on $G$, and $H^m(G;\mathbb X)$
the Sobolev space of functions taking values in the finite-dimensional vector space $\mathbb X$ for $\mathbb X$ being $\mathbb R^d$ or $\mathbb M$, where $\mathbb M$ is the space of all $d\times d$ tensors. The corresponding norm and semi-norm are denoted respectively by
$\Vert\cdot\Vert_{m,G}$ and $|\cdot|_{m,G}$.   Let $(\cdot, \cdot)_G$ be the standard inner product on $L^2(G)$ or $L^2(G;\mathbb X)$.
If $G$ is $\Omega$, we abbreviate
$\Vert\cdot\Vert_{m,G}$, $|\cdot|_{m,G}$ and $(\cdot, \cdot)_G$ by $\Vert\cdot\Vert_{m}$, $|\cdot|_{m}$ and $(\cdot, \cdot)$,
respectively. Let $H_0^m(G)$ be the closure of $C_{0}^{\infty}(G)$ with
respect to the norm $\Vert\cdot\Vert_{m,G}$.
Let $\mathbb P_m(G)$ stand for the set of all
polynomials in $G$ with the total degree no more than $m$,
and $\mathbb P_m(G;\mathbb R^d)$ the vectorial version of $\mathbb P_m(G)$.
As usual, $|G|$ denotes the measure of a given open set $G$.
For any finite set $\mathcal{S}$, denote by $\#\mathcal{S}$ the cardinality of $\mathcal{S}$.

We partition the domain $\Omega$ into a family of shape regular simplicial grids $\mathcal {T}_h$  (cf.\ \cite{BrennerScott2008,Ciarlet1978}) with $h:=\max\limits_{K\in \mathcal {T}_h}h_K$
and $h_K:=\mbox{diam}(K)$.
Let $\mathcal{F}_h$ be the union of all $d-1$ dimensional faces
of $\mathcal {T}_h$, $\mathcal{F}^i_h$ the union of all
interior $d-1$ dimensional faces of the triangulation $\mathcal {T}_h$, and $\mathcal{F}^{\partial}_h:=\mathcal{F}_h\backslash\mathcal{F}^i_h$. Similarly, let $\mathcal{E}_h$ be the union of all $d-2$ dimensional faces
of $\mathcal {T}_h$, $\mathcal{E}^i_h$ the union of all
interior $d-2$ dimensional faces of the triangulation $\mathcal {T}_h$, and $\mathcal{E}^{\partial}_h:=\mathcal{E}_h\backslash\mathcal{E}^i_h$.
Set 
\[
  \mathcal{F}^i(K):=\{F\in\mathcal{F}^i_h: F\subset\partial K\}, \quad \mathcal{F}^{\partial}(K):=\{F\in\mathcal{F}_h^{\partial}: F\subset\partial K\},
\]
\[
  \mathcal{E}(K):=\{e\in\mathcal{E}_h: e\subset\partial K\}.
\]
For each $K\in\mathcal{T}_h$, denote by $n_K$ the
unit outward normal to $\partial K$. Without causing any confusion, we will abbreviate $n_K$ as $n$ for simplicity.
For each $F\in\mathcal{F}_h$,
denote by $h_F$ its diameter and fix a unit normal vector $n_F$ such that $n_F=n_K$ if $F\in\mathcal{F}^{\partial}(K)$.
In two dimensions, i.e. $d=2$, we use $t_F$ to denote the unit tangential vector of $F$ if $F\in\mathcal{F}^{\partial}_h$, and abbreviate it as $t$ for simplicity.
For $s\geq1$, define
\[
  H^s(\mathcal T_h):=\{v\in L^2(\Omega): v|_K\in H^s(K)\quad\forall~K\in\mathcal T_h\}.
\]
For any $v\in H^s(\mathcal T_h)$, define the broken $H^s$ norm and seminorm
\begin{gather*}
  \|v\|_{s,h}^2:=\sum_{K \in \mathcal {T}_h} \|v\|_{s, K}^2, \quad |v|_{s,h}^2:=\sum_{K \in \mathcal {T}_h} |v|_{s, K}^2.
\end{gather*}
For any $v\in H^2(\mathcal T_h)$, introduce some other discrete norms
\begin{gather*}
  \interleave v\interleave_{2,h}^2:= |v|_{2, h}^2+\sum_{F \in \mathcal{F}_{h}^{\partial}} h_F^{-1} \| \partial_nv\|_{0,F}^2 ,\\
  \| v \|_{\varepsilon ,h}^2:=\varepsilon^2 |v|_{2,h}^2+|v|_{1,h}^2, \quad \interleave v\interleave_{\varepsilon ,h}^2:=\varepsilon^2\interleave v\interleave_{2,h}^2+|v|_{1,h}^2 .
\end{gather*}
For any face $F\in\mathcal{F}_h$, set
\[
  \partial^{-1}F :=\{K\in\mathcal{T}_h: F\subset \partial K\}, \quad  \omega_F:=\textrm{interior}\left(\bigcup_{K\in\partial^{-1}F}\overline{K}\right).
\]
For any simplex $K\in\mathcal{T}_h$, denote
\[
  \mathcal T_K:=\{K'\in\mathcal T_h: \overline{K'}\cap\overline{K}\neq\varnothing\}, \quad  \omega_K:=\textrm{interior}\left(\bigcup_{K'\in\mathcal T_K}\overline{K'}\right),
\]
\[
  \omega_{K}^2:=\textrm{interior}\left(\bigcup\{\overline{\omega_{K'}}\in\mathcal T_h: \overline{K'}\cap\overline{K}\neq\varnothing\}\right).
\]
Discrete differential operators $\nabla_h$, $\curl_h$ and $\div_h$ are defined as
the elementwise counterparts of $\nabla$, $\curl$ and $\div$ associated with $\mathcal{T}_h$ respectively.
Throughout this paper, we also use
``$\lesssim\cdots $" to mean that ``$\leq C\cdots$", where
$C$ is a generic positive constant independent of $h$ and the
parameter $\varepsilon$,
which may take different values at different appearances.

Moreover, we introduce averages and jumps on $d-1$ dimensional faces as in \cite{HuangHuangHan2010}.
Consider two adjacent simplices $K^+$ and $K^-$ sharing an interior face $F$.
Denote by $n^+$ and $n^-$ the unit outward normals
to the common face $F$ of the simplices $K^+$ and $K^-$, respectively.
For a scalar-valued or vector-valued function $v$, write $v^+:=v|_{K^+}$ and $v^-
  :=v|_{K^-}$.   Then define the average and jump on $F$ as
follows:
\[
  \{v\}:=\frac{1}{2}(v^++v^-),
  \quad  \llbracket v\rrbracket:=v^+n_F\cdot n^++v^-n_F\cdot n^-.
\]
On a face $F$ lying on the boundary $\partial\Omega$, the above terms are
defined by
\[
  \{v\}:=v, \quad \llbracket v\rrbracket
  :=vn_F\cdot n.
\]

Associated with the partition $\mathcal{T}_h$, the global Morley-Wang-Xu (MWX) element space $\widetilde V_h$ consists of all piecewise quadratic functions
on $\mathcal{T}_h$ such that, their integral average over each $(d-2)$-dimensional face
of elements in $\mathcal{T}_h$ are continuous, and their normal derivatives are continuous at
the barycentric point of each $(d-1)$-dimensional face of elements in $\mathcal{T}_h$ (cf. \cite{WangXu2006, Morley1968, WangXu2013}).
And
define
\[
  V_{h} :=\left\{v\in \widetilde V_{h}: \int_{e}v\dd s
  =0\quad \forall
  \,e\in\mathcal{E}^{\partial}_h\right\},
\]
\[
  V_{h0} :=\left\{v\in V_{h}: \int_{F}\partial_nv\dd s
  =0\quad \forall
  \,F\in\mathcal{F}^{\partial}_h\right\}.
\]
Notice that we do not impose the boundary condition $\int_{F}\partial_nv\dd s=0$ in the finite element space $V_h$.
Due to Lemma~4 in \cite{WangXu2006}, we have
\begin{equation}\label{eq:mwxweakcontinuity}
  \int_{F}\llbracket \nabla v_h\rrbracket\dd s=0\quad\forall~v_h\in V_h, F\in\mathcal F_h^i,
\end{equation}
\begin{equation}\label{eq:mwxweakcontinuity-1}
  \int_{F}\llbracket \nabla_F v_h\rrbracket\dd s=0\quad\forall~v_h\in V_h, F\in\mathcal F_h,
\end{equation}
\begin{equation}\label{eq:mwxweakcontinuity0}
  \int_{F}\llbracket \nabla v_h\rrbracket\dd s=0\quad\forall~v_h\in V_{h0}, F\in\mathcal F_h,
\end{equation}
where the surface gradient $\nabla_F v_h:=\nabla v_h - \partial_{n_F}v_hn_F$.

\subsection{Connection operators}
In this subsection we will introduce some operators to connect the Lagrange element space and the MWX element space for analysis.
Let the Lagrange element space
\[
  W_{h} :=\left\{v\in H_0^{1}(\Omega): v|_K\in P_{\ell}(K)
  \quad \forall
  \,K\in\mathcal{T}_h\right\}
\]
with $\ell=1$ or $2$.
Define a connection operator $E_h^L: V_h\to W_h$ with $\ell=2$ as follows:
Given $v_h\in V_h$, $E_h^Lv_h\in W_h$ is determined by
\[
  N(E_h^Lv_h):=\frac{1}{\#\mathcal T_N}\sum_{K\in\mathcal T_N}N(v_h|_K)
\]
for each interior degree of freedom $N$ of the space $W_h$,
where $\mathcal T_N\subset \mathcal T_h$ denotes the set of simplices sharing the degree of freedom $N$.
By the weak continuity of $V_{h0}$ and $V_h$ and the techniques adopted in \cite{Wang2001,BrennerSung2005}, we have for any $s=1, 2$, $0\leq m\leq s$ and $j=0,1,2$ that
\begin{equation*}
  |v_h-E_h^Lv_h|_{m,K}\lesssim h_K^{s-m}|v_h|_{s, \omega_K}\quad\forall~v_h\in V_{h0},
\end{equation*}
\begin{equation*}
  \|v_h-E_h^Lv_h\|_{0,K}+h_K|v_h-E_h^Lv_h|_{1,K}\lesssim h_K|v_h|_{1, \omega_K}\quad\forall~v_h\in V_{h},
\end{equation*}
\begin{equation*}
  |v_h-E_h^Lv_h|_{j,K}\lesssim h_K^{2-j}\Big(|v_h|_{2, \omega_K}+\sum_{K'\in\mathcal T_K}\sum_{F \in \mathcal{F}^{\partial}(K')} h_F^{-1/2} \| \partial_nv_h\|_{0,F}\Big)\;\;\forall~v_h\in V_{h}
\end{equation*}
for each $K\in\mathcal T_h$.
Then we get
\begin{equation}\label{eq:EhLerror1}
  |v_h-E_h^Lv_h|_{1,h}\lesssim \min\{|v_h|_{1,h},h^{1/2}|v_h|_{1,h}^{1/2}|v_h|_{2,h}^{1/2}, h|v_h|_{2,h}\}\quad\forall~v_h\in V_{h0},
\end{equation}
\begin{equation}\label{eq:EhLerror2}
  |v_h-E_h^Lv_h|_{1,h}\lesssim \min\{|v_h|_{1,h},h^{1/2}|v_h|_{1,h}^{1/2}\interleave v\interleave_{2,h}^{1/2}, h\interleave v\interleave_{2,h}\}\quad\forall~v_h\in V_h.
\end{equation}

To define interpolation operators later, we also need
another two connection operators $E_h: W_h\to V_{h}$ and $E_{h0}: W_h\to V_{h0}$. For any $v_h\in W_h$, $E_hv_h\in V_{h}$ is determined by
\[
  \int_e E_hv_h\dd s=\int_ev_h\dd s\quad\forall~e\in\mathcal E_h^i,
\]
\[
  \int_F\partial_{n_F}( E_hv_h)\dd s=\int_F\{\partial_{n_F}v_h\}\dd s\quad\forall~F\in\mathcal F_h.
\]
And $E_{h0}v_h\in V_{h0}$ is determined by
\[
  \int_e E_{h0}v_h\dd s=\int_ev_h\dd s\quad\forall~e\in\mathcal E_h^i,
\]
\[
  \int_F\partial_{n_F}( E_{h0}v_h)\dd s=\int_F\{\partial_{n_F}v_h\}\dd s\quad\forall~F\in\mathcal F_h^i.
\]

\subsection{Interpolation operators}

Let $I_h^{SZ}$ be the Scott-Zhang interpolation operator \cite{ScottZhang1990} from $H_0^1(\Omega)$ onto $W_h$ with $\ell=2$.
For any $1\leq s\leq 3$ and $0\leq m\leq s$, it holds (cf. \cite[(4.3)]{ScottZhang1990})
\begin{equation}\label{eq:szerror}
  |v-I_h^{SZ}v|_{m,K}\lesssim h_K^{s-m}|v|_{s, \omega_K}\quad\forall~v\in H_0^1(\Omega)\cap H^s(\Omega),\; K\in\mathcal T_h.
\end{equation}
Then define two quasi-interpolation operators $I_h: H_0^1(\Omega)\to V_{h}$ and $I_{h0}: H_0^1(\Omega)\to V_{h0}$ as
\[
  I_h:=E_hI_h^{SZ},
  \quad
  I_{h0}:=E_{h0}I_h^{SZ}.
\]
Next we will derive the error estimates of the interpolation operators $I_h$ and $I_{h0}$ following the argument in \cite{GuzmanLeykekhmanNeilan2012}.
\begin{lemma}
  Let $2\leq s\leq 3$ and $0\leq m\leq s$. We have
  \begin{equation}\label{eq:Ih0Error1}
    |v-I_{h0}v|_{m,h}\lesssim h^{s-m}|v|_s\quad\forall~v\in H^s(\Omega)\cap H_0^2(\Omega),
  \end{equation}
  \begin{equation}\label{eq:Ih0Error2}
    |v-I_{h0}v|_{1, h}^2\lesssim h|v|_1|v|_2\quad\forall~v\in H_0^2(\Omega),
  \end{equation}
  \begin{equation}\label{eq:IhError1}
    |v-I_hv|_{m,h}\lesssim h^{s-m}|v|_s\quad\forall~v\in H^s(\Omega)\cap H_0^1(\Omega),
  \end{equation}
  \begin{equation}\label{eq:IhError2}
    |v-I_hv|_{1, h}^2\lesssim h|v|_1|v|_2\quad\forall~v\in H^2(\Omega)\cap H_0^1(\Omega),
  \end{equation}
  \begin{equation}\label{eq:IhError21}
    |v-I_hv|_{1, h}\lesssim |v|_1\quad\forall~v\in H_0^1(\Omega).
  \end{equation}
\end{lemma}
\begin{proof}
  We only prove the inequalities \eqref{eq:IhError1}-\eqref{eq:IhError2}. The inequalities \eqref{eq:Ih0Error1}-\eqref{eq:Ih0Error2} and~\eqref{eq:IhError21} can be achieved by the same argument.
  Take any $K\in\mathcal T_h$.
  By the definition of $E_h$, we have
  \[
    \int_e(I_h^{SZ}v-E_hI_h^{SZ}v)|_K\dd s=0\quad\forall~e\in\mathcal E(K),
  \]
  \[
    \int_F\partial_{n_F}((I_h^{SZ}v-E_hI_h^{SZ}v)|_K)\dd s=\frac{n_F\cdot n_K}{2}\int_F\llbracket\partial_{n_F}(I_h^{SZ}v)\rrbracket\dd s\quad\forall~F\in\mathcal F^i(K),
  \]
  \[
    \int_F\partial_{n_F}((I_h^{SZ}v-E_hI_h^{SZ}v)|_K)\dd s=0\quad\forall~F\in\mathcal F^{\partial}(K).
  \]
  Applying the inverse inequality, scaling argument and Cauchy-Schwarz inequality, it follows
  \begin{align}
    |I_h^{SZ}v-E_hI_h^{SZ}v|_{m, K} & \leq  h_K^{-m}\|I_h^{SZ}v-E_hI_h^{SZ}v\|_{0, K} \notag                                                                                      \\
                                    & \lesssim h_K^{2-d/2-m}\sum_{F\in\mathcal F^i(K)}\left|\int_F\llbracket\partial_{n_F}(I_h^{SZ}v)\rrbracket\dd s\right| \notag                \\
                                    & \lesssim h_K^{3/2-m}\sum_{F\in\mathcal F^i(K)}\left\|\llbracket\partial_{n_F}(I_h^{SZ}v-v)\rrbracket\right\|_{0,F},\label{eq:temp201806091}
  \end{align}
  which together with the trace inequality and \eqref{eq:szerror} implies
  \begin{align*}
    |I_h^{SZ}v-E_hI_h^{SZ}v|_{m, K} & \lesssim h_K^{1-m} \sum_{F\in\mathcal F^i(K)}\sum_{K'\in\partial^{-1}F}(|v-I_h^{SZ}v|_{1,K'}+h_K|v-I_h^{SZ}v|_{2,K'}) \\
                                    & \lesssim h_K^{s-m}|v|_{s,\omega_{K}^2}.
  \end{align*}
  Employing \eqref{eq:szerror}, we get
  \[
    |v-I_hv|_{m, K}\leq |v-I_h^{SZ}v|_{m, K}+|I_h^{SZ}v-E_hI_h^{SZ}v|_{m, K}\lesssim h_K^{s-m}|v|_{s,\omega_{K}^2},
  \]
  which indicates \eqref{eq:IhError1}.

  Let $\hat K$ be the reference element of $K$, then it holds (cf. \cite[Theorem~1.5.1.10]{Grisvard1985})
  \begin{equation}\label{eq:temp201806104}
    \|w\|_{0,\partial \hat{K}}^2\lesssim \|w\|_{0,\hat K}\|w\|_{1,\hat K}\quad\forall~w\in H^1(\hat K).
  \end{equation}
  We obtain from \eqref{eq:temp201806091}, scaling argument and \eqref{eq:szerror} that
  \begin{align*}
     & \quad\;\;|I_h^{SZ}v-E_hI_h^{SZ}v|_{1, K}^2\lesssim h_K\sum_{F\in\mathcal F^i(K)}\left\|\llbracket\partial_{n_F}(I_h^{SZ}v-v)\rrbracket\right\|_{0,F}^2 \\
     & \lesssim \sum_{F\in\mathcal F^i(K)}\sum_{K'\in\partial^{-1}F}|v-I_h^{SZ}v|_{1,K'}(|v-I_h^{SZ}v|_{1,K'}+h_K|v-I_h^{SZ}v|_{2,K'})                        \\
     & \lesssim h_K|v|_{1,\omega_{K}^2}|v|_{2,\omega_{K}^2}.
  \end{align*}
  Finally we achieve \eqref{eq:IhError2} from the last inequality, triangle inequality and \eqref{eq:szerror}.
\end{proof}

Let $u^0\in H_0^1(\Omega)$ be the solution of the Poisson equation
\begin{equation}\label{poisson}
  \begin{cases}
    -\Delta u^0=f \quad\;\; \textrm{in}~\Omega, \\
    u^0=0 \quad\quad\quad\; \textrm{on}~\partial\Omega.
  \end{cases}
\end{equation}
Since the domain $\Omega$ is convex, we have the following regularities \cite{NilssenTaiWinther2001,GuzmanLeykekhmanNeilan2012,Grisvard1985}
\begin{equation}\label{eq:regularity1}
  |u|_2+\varepsilon |u|_3\lesssim \varepsilon^{-1/2}\|f\|_0,
\end{equation}
\begin{equation}\label{eq:regularity2}
  |u-u^0|_1\lesssim \varepsilon^{1/2}\|f\|_0,
\end{equation}
\begin{equation}\label{eq:regularity3}
  \|u^0\|_2\lesssim \|f\|_0.
\end{equation}

\begin{lemma}
  We have
  \begin{equation}\label{eq:IhError4}
    \interleave u-I_hu\interleave_{\varepsilon,h}\lesssim (\varepsilon h+ h^{2})|u|_3.
  \end{equation}
  If $u^0\in H_0^1(\Omega)\cap H^s(\Omega)$ with $2\leq s\leq 3$,  it holds
  \begin{equation}\label{eq:IhError3}
    \interleave u-I_hu\interleave_{\varepsilon,h}\lesssim \min\{\varepsilon^{1/2}, h^{1/2}\}\|f\|_0 +h^{s-1}|u^0|_{s}.
  \end{equation}
\end{lemma}
\begin{proof}
  The inequality \eqref{eq:IhError4} follows from \eqref{eq:IhError1} and the trace inequality.
  Next we prove \eqref{eq:IhError3}.
  By \eqref{eq:IhError21}, we have
  \[
    |u-u^0-I_h(u-u^0)|_{1,h}\lesssim |u-u^0|_1.
  \]
  Due to \eqref{eq:IhError2}, it follows
  \[
    |u-u^0-I_h(u-u^0)|_{1,h}^2\lesssim h|u-u^0|_1|u-u^0|_2.
  \]
  Combining the last two inequality, we get from \eqref{eq:regularity1}-\eqref{eq:regularity3} that
  \begin{align*}
    |u-u^0-I_h(u-u^0)|_{1,h}^2 & \lesssim \min\{|u-u^0|_1^2, h|u-u^0|_1|u-u^0|_2\} \\
                               & \lesssim \min\{\varepsilon, h\}\|f\|_0^2.
  \end{align*}
  Then using the triangle inequality and \eqref{eq:IhError1}, we acquire
  \begin{align}
    |u-I_hu|_{1,h} & \leq |u-u^0-I_h(u-u^0)|_{1,h}+|u^0-I_hu^0|_{1,h} \notag                                       \\
                   & \lesssim \min\{\varepsilon^{1/2}, h^{1/2}\}\|f\|_0 +h^{s-1}|u^0|_{s}.\label{eq:temp201806107}
  \end{align}
  Applying \eqref{eq:IhError1} and \eqref{eq:regularity1} again, we obtain
  \[
    |u-I_hu|_{2,h}\lesssim |u|_2\lesssim \varepsilon^{-1/2}\|f\|_0,\; |u-I_hu|_{2,h}^2\lesssim h|u|_2|u|_3\lesssim h\varepsilon^{-2}\|f\|_0^2.
  \]
  Thus
  \begin{equation}\label{eq:temp201806111}
    \varepsilon|u-I_hu|_{2,h}\lesssim \min\{\varepsilon^{1/2}, h^{1/2}\}\|f\|_0.
  \end{equation}
  By the trace inequality, \eqref{eq:IhError1} and \eqref{eq:regularity1},
  \begin{align*}
    \sum_{F \in \mathcal{F}_{h}^{\partial}}h_F^{-1}\|\partial_{n}(u-I_hu)\|_{0,F}^2 & \lesssim \sum_{K\in\mathcal T_h}(h_K^{-2}|u-I_hu|_{1,K}^2+|u-I_hu|_{2,K}^2)                    \\
                                                                                    & \lesssim \min\{|u|_2^2, h|u|_2|u|_3\}\lesssim \varepsilon^{-2}\min\{\varepsilon, h\}\|f\|_0^2.
  \end{align*}
  Finally we derive \eqref{eq:IhError3} from \eqref{eq:temp201806107}-\eqref{eq:temp201806111} and the last inequality.
\end{proof}

\begin{lemma}
  We have
  \begin{equation}\label{eq:Ih0Error4}
    \|u-I_{h0}u\|_{\varepsilon,h}\lesssim (\varepsilon h+ h^{2})|u|_3,
  \end{equation}
  \begin{equation}\label{eq:Ih0Error3}
    \|u-I_{h0}u\|_{\varepsilon,h}\lesssim h^{1/2}\|f\|_0.
  \end{equation}
\end{lemma}
\begin{proof}
  The inequality \eqref{eq:Ih0Error4} is the immediate result of \eqref{eq:Ih0Error1}.
  Applying the trace inequality \eqref{eq:temp201806104}, \eqref{eq:szerror} and \eqref{eq:regularity1}-\eqref{eq:regularity3},
  \[
    \sum_{F\in\mathcal F_h^{\partial}}h_F\|\partial_n(I_{h}^{SZ}(u-u^0)-(u-u^0))\|_{0,F}^2
    \lesssim h|u-u^0|_{1}|u-u^0|_{2}\lesssim h\|f\|_0^2.
  \]
  Using \eqref{eq:szerror} and \eqref{eq:regularity3} again,
  \begin{align*}
             & \sum_{F\in\mathcal F_h^{\partial}}h_F\|\partial_n(I_{h}^{SZ}u-u)\|_{0,F}^2                                                                                              \\
    \lesssim & \sum_{F\in\mathcal F_h^{\partial}}h_F\|\partial_n(I_{h}^{SZ}(u-u^0)-(u-u^0))\|_{0,F}^2 + \sum_{F\in\mathcal F_h^{\partial}}h_F\|\partial_n(I_{h}^{SZ}u^0-u^0)\|_{0,F}^2 \\
    \lesssim & h\|f\|_0^2+h^2|u^0|_2^2\lesssim h\|f\|_0^2.
  \end{align*}
  By the definitions of $I_{h0}$ and $I_{h}$, it follows
  \begin{align*}
    |I_{h}u-I_{h0}u|_{1,h}^2 & \lesssim\sum_{F\in\mathcal F_h^{\partial}}h_F\|\partial_n(I_{h}^{SZ}u)\|_{0,F}^2=\sum_{F\in\mathcal F_h^{\partial}}h_F\|\partial_n(I_{h}^{SZ}u-u)\|_{0,F}^2\lesssim h\|f\|_0^2.
  \end{align*}
  On the other side, we get from \eqref{eq:szerror} and \eqref{eq:regularity1} that
  \begin{align*}
    |I_{h}u-I_{h0}u|_{2,h}^2 & \lesssim\sum_{F\in\mathcal F_h^{\partial}}h_F^{-1}\|\partial_n(I_{h}^{SZ}u)\|_{0,F}^2=\sum_{F\in\mathcal F_h^{\partial}}h_F^{-1}\|\partial_n(I_{h}^{SZ}u-u)\|_{0,F}^2 \\
                             & \lesssim h|u|_2|u|_3\lesssim \varepsilon^{-2}h\|f\|_0^2.
  \end{align*}
  Thus we obtain from the last two inequalities
  \[
    \|I_{h}u-I_{h0}u\|_{\varepsilon,h}\lesssim h^{1/2}\|f\|_0,
  \]
  which combined with \eqref{eq:IhError3} indicates \eqref{eq:Ih0Error3}.
\end{proof}

\section{Morley-Wang-Xu Element Method}

We will propose an MWX element method for the fourth order elliptic singular perturbation problem
\eqref{fourthorderpertub} in this section.

\subsection{Morley-Wang-Xu Element Method}
To present the MWX element method, we need the $H^1$-orthogonal projection $P_h: H^1(\mathcal T_h)\to W_h$: given $v_h\in H^1(\mathcal T_h)$, $P_hv_h\in W_h$ is determined by
\[
  (\nabla P_hv_h, \nabla \chi_h)= (\nabla_hv_h, \nabla \chi_h)\quad\forall~\chi_h\in W_h.
\]
It is well-known that for $s\geq 1$  (cf. \cite{Ciarlet1978,BrennerScott2008})
\begin{equation}\label{eq:Pherror}
  |v-P_hv|_1\lesssim h^{\min\{s-1,\ell\}}\|v\|_s\quad\forall~v\in H_0^1(\Omega)\cap H^s(\Omega).
\end{equation}

We propose the following MWX element method for problem~\eqref{fourthorderpertub}: find $u_{h0} \in V_{h0} $ such that
\begin{equation}\label{mmwx0}
  \varepsilon^2 a_h(u_{h0},v_h)+b_h(u_{h0},v_h)=(f,P_hv_h) \quad \forall~v_h \in V_{h0},
\end{equation}
where
\[
  a_{h}(u_{h0}, v_h):= (\nabla_h^{2}u_{h0}, \nabla_h^{2}v_h),\quad b_{h}(u_{h0}, v_h):= (\nabla_h u_{h0}, \nabla_h v_h).
\]

We use the simplest MWX element to approximate the exact solution in the discrete method~\eqref{mmwx0}.
Compared the standard nonconforming finite element method, we only replace the right hand side term $(f, v_h)$ by $(f, P_hv_h)$,
thus the MWX element method~\eqref{mmwx0} possesses a sparser stiffness matrix than those of the discontinuous Galerkin methods.

\subsection{Error Estimates}
Using Cauchy-Schwarz inequality and \eqref{eq:Ih0Error4}-\eqref{eq:Ih0Error3}, we have following error estimates for $I_{h0}$.
\begin{lemma}
  We have for any $v_h\in V_{h0}$
  \begin{equation}\label{eq:Ih0error5}
    \varepsilon^2 a_h(I_{h0}u-u,v_h)+b_h(I_{h0}u-u,v_h)\lesssim (\varepsilon h+ h^{2})|u|_3\|v_h\|_{\varepsilon, h},
  \end{equation}
  \begin{equation}\label{eq:Ih0error6}
    \varepsilon^2 a_h(I_{h0}u-u,v_h)+b_h(I_{h0}u-u,v_h)\lesssim h^{1/2}\|f\|_0\|v_h\|_{\varepsilon, h}.
  \end{equation}
\end{lemma}

\begin{lemma}\label{lem:consistcyerror0}
  We have for any $v_h\in V_{h0}$
  \begin{equation}\label{eq:consistcyerror0-1}
    \varepsilon^2 a_h(u,v_h)+\varepsilon^2(\div\nabla^2u, \nabla E_h^{L}v_h)\lesssim \varepsilon\min\{\varepsilon, (\varepsilon h)^{1/2}, h\}|u|_3\|v_h\|_{\varepsilon, h},
  \end{equation}
  \begin{equation}\label{eq:consistcyerror0-2}
    \varepsilon^2 a_h(u,v_h)+\varepsilon^2(\div\nabla^2u, \nabla E_h^{L}v_h)\lesssim \min\{\varepsilon^{1/2}, h^{1/2}\}\|f\|_0\|v_h\|_{\varepsilon, h}.
  \end{equation}
\end{lemma}
\begin{proof}
  We get from integration by parts and \eqref{eq:mwxweakcontinuity0} that
  \begin{align*}
      & a_h(u,v_h)+(\div\nabla^2u, \nabla_hv_h)                                                                                                  \\
    = & \sum_{K\in\mathcal T_h}((\nabla^2u)n, \nabla_hv_h)_{\partial K}
    =\sum_{F\in\mathcal F_h}((\nabla^2u)n_F, \llbracket\nabla_hv_h\rrbracket)_F                                                                  \\
    = & \sum_{F\in\mathcal F_h}((\nabla^2u)n_F-Q_0^F((\nabla^2u)n_F), \llbracket\nabla_hv_h\rrbracket)_F                                         \\
    = & \sum_{F\in\mathcal F_h}((\nabla^2u)n_F-Q_0^F((\nabla^2u)n_F), \llbracket\nabla_hv_h\rrbracket-Q_0^F(\llbracket\nabla_hv_h\rrbracket))_F,
  \end{align*}
  where $Q_0^F$ is the $L^2$-orthogonal projection onto the constant space on face $F$.
  By the error estimate of $Q_0^F$ (cf. \cite{Ciarlet1978,BrennerScott2008}) and the inverse inequality, we have
  \[
    a_h(u,v_h)+(\div\nabla^2u, \nabla_hv_h)
    \lesssim |u|_3\min\left\{|v_h|_{1,h}, h^{1/2}|v_h|_{1,h}^{1/2}|v_h|_{2,h}^{1/2}, h|v_h|_{2,h}\right\}.
  \]
  On the other side, it follows from \eqref{eq:EhLerror1} that
  \begin{align*}
    (\div\nabla^2u, \nabla_h(E_h^{L}v_h-v_h)) & \lesssim |u|_3|E_h^Lv_h-v_h|_{1,h}                                                                     \\
                                              & \lesssim |u|_3\min\left\{|v_h|_{1,h}, h^{1/2}|v_h|_{1,h}^{1/2}|v_h|_{2,h}^{1/2}, h|v_h|_{2,h}\right\}.
  \end{align*}
  Combining the last two inequalities gives
  \[
    a_h(u,v_h)+(\div\nabla^2u, \nabla E_h^{L}v_h)
    \lesssim  |u|_3\min\left\{|v_h|_{1,h}, h^{1/2}|v_h|_{1,h}^{1/2}|v_h|_{2,h}^{1/2}, h|v_h|_{2,h}\right\}.
  \]
  Finally we derive \eqref{eq:consistcyerror0-1}-\eqref{eq:consistcyerror0-2} from \eqref{eq:regularity1}.
\end{proof}

\begin{lemma}\label{lem:consistcyerror1}
  It holds for any  $v_h\in V_{h0}$
  \begin{equation}\label{eq:nonconformestimate3}
    b_h(u, v_h-E_h^Lv_h)- (f, P_hv_h-E_h^Lv_h)\lesssim (\varepsilon^{-1/2}+1)h\|f\|_0\|v_h\|_{\varepsilon, h}.
  \end{equation}
  If $u^0\in H^s(\Omega)$ with $2\leq s\leq 3$, it holds for any $v_h\in V_{h0}$
  \begin{align}
             & b_h(u, v_h-E_h^Lv_h)- (f, P_hv_h-E_h^Lv_h) \notag                                                                                               \\
    \lesssim & \left(\min\{\varepsilon^{1/2}, h^{1/2}\}\|f\|_0 +h^{\min\{s-1,\ell\}}\|u^0\|_{s}\right)\|v_h\|_{\varepsilon, h}. \label{eq:nonconformestimate2}
  \end{align}
\end{lemma}
\begin{proof}
  Since $P_hv_h-E_h^Lv_h\in H_0^1(\Omega)$, we get from \eqref{poisson}, integration by parts and the definition of $P_h$ that
  \begin{align*}
    (f, P_hv_h-E_h^Lv_h) & =(\nabla u^0, \nabla(P_hv_h-E_h^Lv_h))=(\nabla u^0, \nabla P_hv_h)-(\nabla u^0, \nabla E_h^Lv_h) \\
                         & =(\nabla P_hu^0, \nabla_hv_h)-(\nabla u^0, \nabla E_h^Lv_h)                                      \\
                         & =(\nabla (P_hu^0-u^0), \nabla_h v_h)+(\nabla u^0, \nabla_h(v_h-E_h^Lv_h)).
  \end{align*}
  Thus we have
  \begin{align*}
      & b_h(u, v_h-E_h^Lv_h)-(f, P_hv_h-E_h^Lv_h)                                    \\
    = & (\nabla(u-u^0), \nabla_h(v_h-E_h^Lv_h))+(\nabla (u^0-P_hu^0), \nabla_h v_h).
  \end{align*}
  Adopting Cauchy-Schwarz inequality and \eqref{eq:regularity2}, it holds
  \begin{align*}
             & b_h(u, v_h-E_h^Lv_h)-(f, P_hv_h-E_h^Lv_h)                               \\
    \lesssim & |u-u^0|_{1}|v_h-E_h^Lv_h|_{1,h}+|u^0-P_hu^0|_1|v_h|_{1,h}               \\
    \lesssim & \varepsilon^{1/2}\|f\|_0|v_h-E_h^Lv_h|_{1,h}+|u^0-P_hu^0|_1|v_h|_{1,h}.
  \end{align*}
  Thanks to \eqref{eq:EhLerror1}, we get
  \[
    |v_h-E_h^Lv_h|_{1,h}\lesssim  \varepsilon^{-1/2}\min\{\varepsilon^{1/2}, h^{1/2}, \varepsilon^{-1/2}h\}\|v_h\|_{\varepsilon, h}.
  \]
  Hence we obtain
  \begin{align*}
             & b_h(u, v_h-E_h^Lv_h)-(f, P_hv_h-E_h^Lv_h)                                                                                      \\
    \lesssim & \min\{\varepsilon^{1/2}, h^{1/2}, \varepsilon^{-1/2}h\}\|f\|_0\|v_h\|_{\varepsilon, h}+|u^0-P_hu^0|_1\|v_h\|_{\varepsilon, h}.
  \end{align*}
  Therefore we acquire \eqref{eq:nonconformestimate3}-\eqref{eq:nonconformestimate2} from \eqref{eq:Pherror}.
\end{proof}

\begin{theorem}\label{thm:ipmwxpriori1}
  Let $u\in H_0^2(\Omega)$ be the solution of problem~\eqref{fourthorderpertub}, and $u_{h0}\in V_{h0}$ be the discrete solution of the MWX element method~\eqref{mmwx0}.
  Assume $u^0\in H^s(\Omega)$ with $2\leq s\leq 3$.
  We have
  \begin{equation}\label{eq:error0}
    \|u-u_{h0}\|_{\varepsilon, h}\lesssim \varepsilon^{1/2}\|f\|_0 +h^{\min\{s-1,\ell\}}\|u^0\|_{s}+h(\varepsilon+h)|u|_3,
  \end{equation}
  \begin{equation}\label{eq:error1}
    \|u-u_{h0}\|_{\varepsilon, h}\lesssim h\left((\varepsilon^{-1/2}+1)\|f\|_0+(\varepsilon+h)|u|_3\right),
  \end{equation}
  \begin{equation}\label{eq:error2}
    \|u-u_{h0}\|_{\varepsilon, h}\lesssim h^{1/2}\|f\|_0,
  \end{equation}
  \begin{equation}\label{eq:error3}
    \|u^0-u_{h0}\|_{\varepsilon, h}\lesssim \left(\varepsilon^{1/2}+h^{1/2}\right)\|f\|_0.
  \end{equation}
\end{theorem}
\begin{proof}
  Let $v_h=I_{h0}u-u_{h0}$. We obtain from \eqref{fourthorderpertub} that
  \[
    -\varepsilon^2(\div\nabla^2u, \nabla E_h^{L}v_h) + (\nabla u, \nabla E_h^{L}v_h)=(f, E_h^{L}v_h).
  \]
  Then it follows from \eqref{mmwx0} that
  \begin{align*}
      & \varepsilon^2 a_h(u-u_{h0},v_h)+b_h(u-u_{h0},v_h)                                                                      \\
    = & \varepsilon^2 a_h(u,v_h)+b_h(u,v_h)-(f, P_hv_h)                                                                        \\
    = & \varepsilon^2 a_h(u,v_h)+\varepsilon^2(\div\nabla^2u, \nabla E_h^{L}v_h)+b_h(u,v_h-E_h^{L}v_h)-(f, P_hv_h-E_h^{L}v_h).
  \end{align*}
  Hence we acquire from
  \eqref{eq:consistcyerror0-1}-\eqref{eq:consistcyerror0-2} and \eqref{eq:nonconformestimate3}-\eqref{eq:nonconformestimate2} that
  \[
    \varepsilon^2 a_h(u-u_{h0},v_h)+b_h(u-u_{h0},v_h)\lesssim h\left((\varepsilon^{-1/2}+1)\|f\|_0+\varepsilon|u|_3\right)\|v_h\|_{\varepsilon, h},
  \]
  \begin{align*}
             & \varepsilon^2 a_h(u-u_{h0},v_h)+b_h(u-u_{h0},v_h)                                                                \\
    \lesssim & \left(\min\{\varepsilon^{1/2}, h^{1/2}\}\|f\|_0 +h^{\min\{s-1,\ell\}}\|u^0\|_{s}\right)\|v_h\|_{\varepsilon, h}.
  \end{align*}
  Using the triangle inequality and \eqref{eq:Ih0error5}-\eqref{eq:Ih0error6}, we get
  \begin{align*}
             & \varepsilon^2 a_h(I_{h0}u-u_{h0},v_h)+b_h(I_{h0}u-u_{h0},v_h)                                                         \\
    \lesssim & \left(\varepsilon^{1/2}\|f\|_0 +h^{\min\{s-1,\ell\}}\|u^0\|_{s}+h(\varepsilon+h)|u|_3\right)\|v_h\|_{\varepsilon, h},
  \end{align*}
  \begin{align*}
             & \varepsilon^2 a_h(I_{h0}u-u_{h0},v_h)+b_h(I_{h0}u-u_{h0},v_h)                             \\
    \lesssim & h\left((\varepsilon^{-1/2}+1)\|f\|_0+(\varepsilon+h)|u|_3\right)\|v_h\|_{\varepsilon, h},
  \end{align*}
  \[
    \varepsilon^2 a_h(I_{h0}u-u_{h0},v_h)+b_h(I_{h0}u-u_{h0},v_h)\lesssim h^{1/2}\|f\|_0\|v_h\|_{\varepsilon, h}.
  \]
  Thus
  \[
    \|I_{h0}u-u_{h0}\|_{\varepsilon, h}\lesssim \varepsilon^{1/2}\|f\|_0 +h^{\min\{s-1,\ell\}}\|u^0\|_{s}+h(\varepsilon+h)|u|_3,
  \]
  \[
    \|I_{h0}u-u_{h0}\|_{\varepsilon, h}\lesssim h\left((\varepsilon^{-1/2}+1)\|f\|_0+(\varepsilon+h)|u|_3\right),
  \]
  \[
    \|I_{h0}u-u_{h0}\|_{\varepsilon, h}\lesssim h^{1/2}\|f\|_0.
  \]
  Finally we get \eqref{eq:error0}-\eqref{eq:error2} by combining the last three inequalities and \eqref{eq:Ih0Error4}-\eqref{eq:Ih0Error3}.
  The estimate \eqref{eq:error3} is a direct result of \eqref{eq:error2} and \eqref{eq:regularity1}-\eqref{eq:regularity3}.
\end{proof}

\section{Imposing Boundary Condition Using Nitsche's Method}
In the consideration of the boundary layer of problem~\eqref{fourthorderpertub}, we will adjust the MWX element method \eqref{mmwx0} by using Nitsche's method to impose the boundary condition $\partial_nu=0$ weakly in this section, as in \cite{GuzmanLeykekhmanNeilan2012}.

\subsection{Discrete Method}
Through applying the Nitsche's technique, the MWX element method with weakly imposing the  boundary condition is to find $u_h \in V_{h} $ such that
\begin{equation}\label{mmwx}
  \varepsilon^2 \tilde a_h(u_h,v_h)+b_h(u_h,v_h)=(f,P_hv_h) \quad \forall~v_h \in V_{h},
\end{equation}
where
\begin{align*}
  \tilde a_{h}(u_h, v_h):= & (\nabla_h^{2}u_h, \nabla_h^{2}v_h) - \sum_{F \in \mathcal{F}_{h}^{\partial}}(\partial_{nn}^2u_h, \partial_{n}v_h)_F- \sum_{F \in \mathcal{F}_{h}^{\partial}}(\partial_{n}u_h, \partial_{nn}^2v_h)_F \\
                           & +\sum_{F \in \mathcal{F}_{h}^{\partial}}\frac{\sigma}{h_{F}}(\partial_{n}u_h, \partial_{n}v_h)_F
\end{align*}
with $\sigma$ being a positive real number.

\begin{lemma}
  There exists a constant $\sigma_0>0$ depending only on the shape regularity of $\mathcal T_h$ such that for any fixed number $\sigma\geq \sigma_0$, it holds
  \begin{equation}\label{eq:ahcoercive}
    \interleave v_h\interleave_{2,h}^2\lesssim \tilde a_h(v_h, v_h)\quad\forall~v_h\in V_h.
  \end{equation}
\end{lemma}
\begin{proof}
  Due to the Cauchy-Schwarz inequality and inverse inequality, there exists a constant $C>0$ such that
  \begin{align*}
    2\sum_{F \in \mathcal{F}_{h}^{\partial}}(\partial_{nn}^2v_h, \partial_{n}v_h)_F & \leq 2\sum_{F \in \mathcal{F}_{h}^{\partial}}\|\partial_{nn}^2v_h\|_{0,F}\|\partial_{n}v_h\|_{0,F}                        \\
                                                                                    & \leq C|v_h|_{2,h}\left(\sum_{F \in \mathcal{F}_{h}^{\partial}}h_{F}^{-1}\|\partial_{n}v_h\|_{0,F}^2\right)^{1/2}          \\
                                                                                    & \leq \frac{1}{2}|v_h|_{2,h}^2+\frac{1}{2}C^2\sum_{F \in \mathcal{F}_{h}^{\partial}}h_{F}^{-1}\|\partial_{n}v_h\|_{0,F}^2.
  \end{align*}
  Hence
  \begin{align*}
    \tilde a_{h}(v_h, v_h) & = |v_h|_{2,h}^2 - 2\sum_{F \in \mathcal{F}_{h}^{\partial}}(\partial_{nn}^2v_h, \partial_{n}v_h)_F+\sum_{F \in \mathcal{F}_{h}^{\partial}}\frac{\sigma}{h_{F}}\|\partial_{n}v_h\|_{0,F}^2 \\
                           & \geq \frac{1}{2}|v_h|_{2,h}^2+\left(\sigma-\frac{1}{2}C^2\right)\sum_{F \in \mathcal{F}_{h}^{\partial}}h_{F}^{-1}\|\partial_{n}v_h\|_{0,F}^2.
  \end{align*}
  The proof is finished by choosing $\sigma_0=\frac{1}{2}C^2+1$.
\end{proof}
By \eqref{eq:ahcoercive}, we have
\begin{equation}\label{eq:coercivity}
  \interleave v_h \interleave_{\varepsilon ,h}^2\lesssim \varepsilon^2 \tilde a_h(v_h,v_h)+b_h(v_h,v_h)\ \ \forall~v_h \in V_{h}.
\end{equation}
It is obvious that
\[
  \varepsilon^2\tilde a_h(\chi_h, v_h)+b_h(\chi_h, v_h)\lesssim  \interleave\chi_h\interleave_{\varepsilon ,h}\interleave v_h\interleave_{\varepsilon ,h} \quad\forall~\chi_h, v_h\in V_h.
\]
The last two inequalities indicate
the wellposedness of the MWX element method~\eqref{mmwx}.

\subsection{Error Estimates}
In this subsection we will present the error analysis for the discrete method~\eqref{mmwx}.
\begin{lemma}
  It holds
  \begin{equation}\label{eq:Iherror5}
    \varepsilon^2 \tilde a_h(I_{h}u-u,v_h)+b_h(I_{h}u-u,v_h)\lesssim (\varepsilon h+ h^{2})|u|_3\interleave v_h\interleave_{\varepsilon, h}\quad\forall~v_h\in V_{h}.
  \end{equation}
  If $u^0\in H_0^1(\Omega)\cap H^s(\Omega)$ with $2\leq s\leq 3$, we have
  \begin{align}
             & \varepsilon^2 \tilde a_h(I_{h}u-u,v_h)+b_h(I_{h}u-u,v_h)\notag                                                                                         \\
    \lesssim & (\min\{\varepsilon^{1/2}, h^{1/2}\}\|f\|_0 +h^{s-1}|u^0|_{s})\interleave v_h\interleave_{\varepsilon, h}\quad\forall~v_h\in V_{h}. \label{eq:Iherror6}
  \end{align}
\end{lemma}
\begin{proof}
  According to the trace inequality, \eqref{eq:IhError1} and \eqref{eq:regularity1}, it follows
  \begin{align*}
    \sum_{F \in \mathcal{F}_{h}^{\partial}}h_F\|\partial_{nn}^2(I_{h}u-u)\|_{0,F}^2 & \lesssim |I_{h}u-u|_{2,h}(|I_{h}u-u|_{2,h}+h|I_{h}u-u|_{3,h})                                    \\
                                                                                    & \lesssim \min\{h|u|_2|u|_3, h^2|u|_3^2\}\lesssim \min\{\varepsilon^{-2}h\|f\|_0^2, h^2|u|_3^2\}.
  \end{align*}
  Then we get 
  \begin{align*}
    - \sum_{F \in \mathcal{F}_{h}^{\partial}}(\partial_{nn}^2(I_{h}u-u), \partial_{n}v_h)_F
    \leq     & \sum_{F \in \mathcal{F}_{h}^{\partial}}\|\partial_{nn}^2(I_{h}u-u)\|_{0,F}\|\partial_{n}v_h\|_{0,F}                                \\
    \lesssim & \interleave v_h\interleave_{2,h}\left(\sum_{F \in \mathcal{F}_{h}^{\partial}}h_F\|\partial_{nn}^2(I_{h}u-u)\|_{0,F}^2\right)^{1/2} \\
    \lesssim & \varepsilon^{-2}\interleave v_h\interleave_{\varepsilon,h}\min\{h^{1/2}\|f\|_0, \varepsilon h|u|_3\}.
  \end{align*}
  Using the inverse inequality, \eqref{eq:IhError1} and \eqref{eq:regularity1}, we obtain
  \begin{align*}
    - \sum_{F \in \mathcal{F}_{h}^{\partial}}(\partial_{nn}^2(I_{h}u-u), \partial_{n}v_h)_F
    \leq     & \sum_{F \in \mathcal{F}_{h}^{\partial}}\|\partial_{nn}^2(I_{h}u-u)\|_{0,F}\|\partial_{n}v_h\|_{0,F}                \\
    \lesssim & |v_h|_{1,h}\left(\sum_{F \in \mathcal{F}_{h}^{\partial}}h_F^{-1}\|\partial_{nn}^2(I_{h}u-u)\|_{0,F}^2\right)^{1/2} \\
    \lesssim & |v_h|_{1,h}|u|_3\lesssim \varepsilon^{-3/2}\interleave v_h\interleave_{\varepsilon,h}\|f\|_0.
  \end{align*}
  Hence it follows from the last two inequalities that
  \begin{equation}\label{eq:temp201806261}
    - \varepsilon^{2}\sum_{F \in \mathcal{F}_{h}^{\partial}}(\partial_{nn}^2(I_{h}u-u), \partial_{n}v_h)_F\lesssim \interleave v_h\interleave_{\varepsilon,h}\min\{\varepsilon^{1/2}\|f\|_0, h^{1/2}\|f\|_0, \varepsilon h|u|_3\}.
  \end{equation}
  Since
  \begin{align*}
             & \varepsilon^2 \tilde a_h(I_{h}u-u,v_h)+b_h(I_{h}u-u,v_h)                                                                                                                                          \\
    \lesssim & \interleave I_{h}u-u\interleave_{\varepsilon,h}\interleave v_h\interleave_{\varepsilon,h} - \varepsilon^{2}\sum_{F \in \mathcal{F}_{h}^{\partial}}(\partial_{nn}^2(I_{h}u-u), \partial_{n}v_h)_F,
  \end{align*}
  we acquire \eqref{eq:Iherror5} from \eqref{eq:IhError4} and \eqref{eq:temp201806261}, and \eqref{eq:Iherror6} from \eqref{eq:IhError3} and \eqref{eq:temp201806261}.
\end{proof}

Applying the same argument as in Lemma~\ref{lem:consistcyerror0}, from \eqref{eq:mwxweakcontinuity}-\eqref{eq:mwxweakcontinuity-1} and \eqref{eq:EhLerror2} we obtain the following estimates
\begin{equation}\label{eq:consistcyerror-1}
  \varepsilon^2 \tilde a_h(u,v_h)+\varepsilon^2(\div\nabla^2u, \nabla E_h^{L}v_h)\lesssim \varepsilon\min\{\varepsilon, (\varepsilon h)^{1/2}, h\}|u|_3\interleave v_h\interleave_{\varepsilon, h},
\end{equation}
\begin{equation}\label{eq:consistcyerror-2}
  \varepsilon^2 \tilde a_h(u,v_h)+\varepsilon^2(\div\nabla^2u, \nabla E_h^{L}v_h)\lesssim \min\{\varepsilon^{1/2}, h^{1/2}\}\|f\|_0\interleave v_h\interleave_{\varepsilon, h}
\end{equation}
for any $v_h\in V_{h}$.

Applying the same argument as in Lemma~\ref{lem:consistcyerror1}, we acquire
\begin{equation}\label{eq:nonconformestimate3-1}
  b_h(u, v_h-E_h^Lv_h)- (f, P_hv_h-E_h^Lv_h)\lesssim (\varepsilon^{-1/2}+1)h\|f\|_0\interleave v_h\interleave_{\varepsilon, h}
\end{equation}
for any  $v_h\in V_{h}$. And if $u^0\in H^s(\Omega)$ with $2\leq s\leq 3$, it holds 
\begin{align}
           & b_h(u, v_h-E_h^Lv_h)- (f, P_hv_h-E_h^Lv_h) \notag                                                                                                                                              \\
  \lesssim & \left(\min\{\varepsilon^{1/2}, h^{1/2}\}\|f\|_0 +h^{\min\{s-1,\ell\}}\|u^0\|_{s}\right)\interleave v_h\interleave_{\varepsilon, h} \quad\forall~v_h\in V_{h}. \label{eq:nonconformestimate2-1}
\end{align}

\begin{theorem}\label{thm:ipmwxpriori2}
  Let $u\in H_0^2(\Omega)$ be the solution of problem~\eqref{fourthorderpertub}, and $u_{h}\in V_{h}$ be the discrete solution of the MWX element method~\eqref{mmwx}. Assume $u^0\in H_0^1(\Omega)\cap H^s(\Omega)$ with $2\leq s\leq 3$.
  We have
  \begin{equation}\label{eq:error4}
    \interleave u-u_{h}\interleave_{\varepsilon, h}\lesssim h\left((\varepsilon^{-1/2}+1)\|f\|_0+(\varepsilon+h)|u|_3\right),
  \end{equation}
  \begin{equation}\label{eq:error5}
    \interleave u-u_{h}\interleave_{\varepsilon, h}\lesssim \min\{\varepsilon^{1/2}, h^{1/2}\}\|f\|_0 +h^{\min\{s-1,\ell\}}\|u^0\|_{s},
  \end{equation}
  \begin{equation}\label{eq:error6}
    \| u^0-u_h\|_{\varepsilon, h}\lesssim \varepsilon^{1/2}\|f\|_0 +h^{\min\{s-1,\ell\}}\|u^0\|_{s}.
  \end{equation}
\end{theorem}
\begin{proof}
  Let $v_h=I_{h}u-u_{h}$. Adopting the similar argument as in the proof of Theorem~\ref{thm:ipmwxpriori1},
  we get from
  \eqref{eq:consistcyerror-1}-\eqref{eq:consistcyerror-2} and \eqref{eq:nonconformestimate3-1}-\eqref{eq:nonconformestimate2-1} that
  \[
    \varepsilon^2 \tilde a_h(u-u_{h},v_h)+b_h(u-u_{h},v_h)\lesssim h\left((\varepsilon^{-1/2}+1)\|f\|_0+\varepsilon|u|_3\right)\interleave v_h\interleave_{\varepsilon, h},
  \]
  \[
    \varepsilon^2 \tilde a_h(u-u_{h},v_h)+b_h(u-u_{h},v_h)
    \lesssim \left(\min\{\varepsilon^{1/2}, h^{1/2}\}\|f\|_0 +h^{\min\{s-1,\ell\}}\|u^0\|_{s}\right)\interleave v_h\interleave_{\varepsilon, h}.
  \]
  Together with \eqref{eq:Iherror5}-\eqref{eq:Iherror6}, we have
  \begin{align*}
             & \varepsilon^2\tilde a_h(I_{h}u-u_{h},v_h)+b_h(I_{h}u-u_{h},v_h)                                              \\
    \lesssim & h\left((\varepsilon^{-1/2}+1)\|f\|_0+(\varepsilon+h)|u|_3\right)\interleave v_h\interleave_{\varepsilon, h},
  \end{align*}
  \begin{align*}
             & \varepsilon^2\tilde a_h(I_{h}u-u_{h},v_h)+b_h(I_{h}u-u_{h},v_h)                                                                     \\
    \lesssim & \left(\min\{\varepsilon^{1/2}, h^{1/2}\}\|f\|_0 +h^{\min\{s-1,\ell\}}\|u^0\|_{s}\right)\interleave v_h\interleave_{\varepsilon, h}.
  \end{align*}
  Then we obtain from \eqref{eq:coercivity} that
  \[
    \interleave I_{h}u-u_{h}\interleave_{\varepsilon, h}\lesssim h\left((\varepsilon^{-1/2}+1)\|f\|_0+(\varepsilon+h)|u|_3\right),
  \]
  \[
    \interleave I_{h}u-u_{h}\interleave_{\varepsilon, h}\lesssim \min\{\varepsilon^{1/2}, h^{1/2}\}\|f\|_0 +h^{\min\{s-1,\ell\}}\|u^0\|_{s}.
  \]
  Therefore we conclude \eqref{eq:error4} and \eqref{eq:error5} from \eqref{eq:IhError4}-\eqref{eq:IhError3},
  and \eqref{eq:error6} from \eqref{eq:error5} and \eqref{eq:regularity1}-\eqref{eq:regularity3}.
\end{proof}

\section{Equivalent Formulations}
In this section, we will show some equivalent solver-friendly formulations of the MWX element methods~\eqref{mmwx0} and \eqref{mmwx}.

\begin{lemma}\label{lem:equivalence1-0}
  The MWX element method~\eqref{mmwx0}
  is equivalent to find
  $w_h\in W_h$ and $u_{h0}\in V_{h0}$ such that
  \begin{align}
    (\nabla w_h, \nabla \chi_h)                   & =(f, \chi_h)  \quad\quad\quad\quad\; \forall~\chi_h\in W_h, \label{mmwx0equiv1}  \\
    \varepsilon^2 a_h(u_{h0},v_h)+b_h(u_{h0},v_h) & =(\nabla w_h, \nabla_h v_h) \quad\;\; \forall~v_h\in V_{h0}. \label{mmwx0equiv2}
  \end{align}
\end{lemma}
\begin{proof}
  By the definition of the $H^1$-orthogonal projection $P_h$ and \eqref{mmwx0equiv1} with $\chi_h= P_hv_h$, the right hand side of \eqref{mmwx0equiv2}, it follows
  \[
    (\nabla w_h, \nabla_h v_h)=(\nabla w_h, \nabla P_hv_h)=(f, P_hv_h).
  \]
  Therefore the MWX element method~\eqref{mmwx0}
  is equivalent to the discrete method \eqref{mmwx0equiv1}-\eqref{mmwx0equiv2}.
\end{proof}

Similarly, we have the equivalent formulation of the MWX element method~\eqref{mmwx}.
\begin{lemma}\label{lem:equivalence1}
  The MWX element method~\eqref{mmwx}
  is equivalent to find
  $w_h\in W_h$ and $u_{h}\in V_h$ such that
  \begin{align}
    (\nabla w_h, \nabla \chi_h)                    & =(f, \chi_h)  \quad\quad\quad\quad\; \forall~\chi_h\in W_h, \label{mmwxequiv1} \\
    \varepsilon^2 \tilde a_h(u_h,v_h)+b_h(u_h,v_h) & =(\nabla w_h, \nabla_h v_h) \quad\;\; \forall~v_h\in V_h. \label{mmwxequiv2}
  \end{align}
\end{lemma}

In two dimensions, we can further decouple the discrete methods \eqref{mmwx0equiv2} and \eqref{mmwxequiv2} into the discrete methods of two Poisson equations and one Brinkman problem.
To this end, define the vectorial nonconforming $P_1$ element space
\begin{align*}
  V_h^{CR}:=\Big\{v\in L^2(\Omega;\mathbb R^2): & ~v|_K\in\mathbb P_1(K;\mathbb R^2)\textrm{ for each } K\in\mathcal T_h, \int_F\llbracket v\rrbracket\dd s=0 \textrm{ for }              \\
                                                & \quad\; \textrm{ each } F\in\mathcal F_h^i, \textrm{ and } \int_Fv\cdot n\dd s=0 \textrm{ for each } F\in\mathcal F_h^{\partial}\Big\},
\end{align*}
\[
  V_{h0}^{CR}:=\Big\{v\in V_h^{CR}: \int_Fv\dd s=0 \textrm{ for each } F\in\mathcal F_h^{\partial}\Big\}.
\]
And let $\mathcal Q_h\subset L_0^2(\Omega)$ be the piecewise constant space with respect to $\mathcal T_h$, where $L_0^2(\Omega)$ is the subspace of $L^2(\Omega)$ with vanishing mean value.

Due to Theorem 4.1 in \cite{FalkMorley1990}, we have the following relationship between Morley element spaces and vectorial Crouzeix-Raviart element spaces
\begin{equation}\label{eq:morleycr0}
  \curl_h V_{h0}=\{v_h\in V_{h0}^{CR}: \div_h v_h=0\},
\end{equation}
\begin{equation}\label{eq:morleycr}
  \curl_h V_h=\{v_h\in V_h^{CR}: \div_h v_h=0\}.
\end{equation}

\begin{lemma}\label{lem:equivalence2-0}
  In two dimensions,
  the discrete method \eqref{mmwx0equiv2} can be decoupled into
  two Morley element methods of Poisson equation and one nonconforming $P_1$-$P_0$ element method of Brinkman problem, i.e.,
  find $(z_h, \phi_h, p_h, w_h)\in V_{h}\times V_{h0}^{CR}\times \mathcal Q_{h}\times V_{h}$ such that
  \begin{subequations}
    \begin{align}
      (\curl_h z_h, \curl_h v_h)                                                           & =(\nabla w_h, \nabla_h v_h)\quad\, \forall~v_h\in V_{h},  \label{eq:4thpurbmfem0a}           \\
      (\phi_h, \psi_h)+\varepsilon^2(\nabla_h\phi_h, \nabla_h\psi_h) + (\div_h\psi_h, p_h) & =(\curl_h z_h, \psi_h)
      \quad \forall~\psi_h\in  V_{h0}^{CR}, \label{eq:4thpurbmfem0b}                                                                                                                      \\
      (\div_h\phi_h, q_h)                                                                  & = 0\quad\quad\quad\quad\quad\quad\, \forall~q_h\in \mathcal Q_{h},  \label{eq:4thpurbmfem0c} \\
      (\curl_h u_{h0}, \curl_h \chi_h)                                                     & = (\phi_h, \curl_h\chi_h) \quad \forall~\chi_h\in V_{h}. \label{eq:4thpurbmfem0d}
    \end{align}
  \end{subequations}
\end{lemma}
\begin{proof}
  Thanks to \eqref{eq:morleycr0}, it follows from \eqref{eq:4thpurbmfem0c}-\eqref{eq:4thpurbmfem0d} that
  \[
    \phi_h=\curl_h u_{h0} \quad\textrm{ and } \quad u_{h0}\in V_{h0}.
  \]
  For any $v_h\in V_{h0}$, it is apparent that
  \[
    (\nabla_h\curl_h u_{h0}, \nabla_h\curl_h v_h)=a_h(u_{h0}, v_h).
  \]
  Then replacing $\phi_h$ with $\curl_h u_{h0}$ and $\psi_h$ with $\curl_h v_h$ in \eqref{eq:4thpurbmfem0b}, we achieve
  \[
    \varepsilon^2 a_h(u_{h0},v_h)+b_h(u_{h0},v_h) =(\curl_h z_h, \curl_h v_h),
  \]
  which combined with \eqref{eq:4thpurbmfem0a} induces \eqref{mmwx0equiv2}.
\end{proof}

Similarly, we get the decoupling of the discrete method \eqref{mmwxequiv2} based on \eqref{eq:morleycr}.
\begin{lemma}\label{lem:equivalence2}
  In two dimensions,
  the discrete method \eqref{mmwxequiv2} can be decoupled into
  two Morley element methods of Poisson equation and one nonconforming $P_1$-$P_0$ element method of Brinkman problem, i.e.,
  find $(z_h, \phi_h, p_h, w_h)\in V_{h}\times V_h^{CR}\times \mathcal Q_{h}\times V_{h}$ such that
  \begin{subequations}
    \begin{align}
      (\curl_h z_h, \curl_h v_h)                                              & =(\nabla w_h, \nabla_h v_h)\quad\;\; \forall~v_h\in V_{h},  \label{eq:4thpurbmfema}           \\
      (\phi_h, \psi_h)+\varepsilon^2c_h(\phi_h, \psi_h) + (\div_h\psi_h, p_h) & =(\curl_h z_h, \psi_h)
      \quad\;\, \forall~\psi_h\in  V_h^{CR}, \label{eq:4thpurbmfemb}                                                                                                          \\
      (\div_h\phi_h, q_h)                                                     & = 0\quad\quad\quad\quad\quad\quad\;\; \forall~q_h\in \mathcal Q_{h},  \label{eq:4thpurbmfemc} \\
      (\curl_h u_h, \curl_h \chi_h)                                           & = (\phi_h, \curl_h\chi_h) \quad\; \forall~\chi_h\in V_{h}, \label{eq:4thpurbmfemd}
    \end{align}
  \end{subequations}
  where
  \begin{align*}
    c_h(\phi_h, \psi_h):= & (\nabla_h\phi_h, \nabla_h\psi_h) - \sum_{F \in \mathcal{F}_{h}^{\partial}}(\partial_{n}(\phi_h\cdot t), \psi_h\cdot t)_F - \sum_{F \in \mathcal{F}_{h}^{\partial}}(\phi_h\cdot t, \partial_{n}(\psi_h\cdot t))_F \\
                          & +\sum_{F \in \mathcal{F}_{h}^{\partial}}\frac{\sigma}{h_{F}}(\phi_h\cdot t, \psi_h\cdot t)_F.
  \end{align*}
\end{lemma}

Combining Lemma~\ref{lem:equivalence1-0} and Lemma~\ref{lem:equivalence2-0} yields an equivalent discrete method of the MWX element method~\eqref{mmwx0}.
And combining Lemma~\ref{lem:equivalence1} and Lemma~\ref{lem:equivalence2} yields an equivalent discrete method of the MWX element method~\eqref{mmwx}.
\begin{theorem}\label{thm:equivalence3-0}
  In two dimensions,
  the MWX element method~\eqref{mmwx0}
  can be decoupled into \eqref{mmwx0equiv1} and \eqref{eq:4thpurbmfem0a}-\eqref{eq:4thpurbmfem0d}.
  That is, the MWX element method~\eqref{mmwx0} can be decoupled into one Lagrange element method of Poisson equation,
  two Morley element methods of Poisson equation and one nonconforming $P_1$-$P_0$ element method of Brinkman problem.
\end{theorem}
\begin{theorem}\label{thm:equivalence3}
  In two dimensions,
  the MWX element method~\eqref{mmwx} can be decoupled into one Lagrange element method of Poisson equation \eqref{mmwxequiv1},
  two Morley element methods of Poisson equation and one nonconforming $P_1$-$P_0$ element method of Brinkman problem \eqref{eq:4thpurbmfema}-\eqref{eq:4thpurbmfemd}.
\end{theorem}

The decoupling of the fourth order elliptic singular perturbation problem \eqref{fourthorderpertub} into two Poisson equations and one Brinkman problem in the continuous level have been developed in \cite{ChenHuang2018,Gallistl2017}.
The decoupling of the Morley element method of the biharmonic equation into two Morley element methods of Poisson equation and one nonconforming $P_1$-$P_0$ element method of Stokes equation was firstly discovered in \cite{Huang2010}.

When $\varepsilon$ is very small, the stiffness matrix of the MWX element method~\eqref{mmwx0} is very close to the stiffness matrix of the MWX element method for the Poisson equation, which can be efficiently solved by the CG method with AMG as the preconditioner.

When $\varepsilon\eqsim 1$, the equivalences in Theorems~\ref{thm:equivalence3-0}-\ref{thm:equivalence3} will induce efficient and robust Poisson based solvers for the MWX element method~\eqref{mmwx0} and \eqref{mmwx}.
The Lagrange element methods of Poisson equation~\eqref{mmwx0equiv1} and \eqref{mmwxequiv1},
the Morley element methods of Poisson equation~\eqref{eq:4thpurbmfem0a}, \eqref{eq:4thpurbmfem0d}, \eqref{eq:4thpurbmfema} and \eqref{eq:4thpurbmfemd} can be solved by CG method with the auxiliary space preconditioner \cite{Xu1996},
in which the $H^1$ conforming linear element discretization on the same mesh for the
Poisson equation can be adopted as the auxiliary problem.
And the AMG method is used to solve the auxiliary problem.
If $\ell=1$, the Lagrange element methods of Poisson equation~\eqref{mmwx0equiv1} and \eqref{mmwxequiv1} are the linear Lagrange element methods, which can be solved efficiently by CG method using the classical AMG method as the preconditioner.
As for the nonconforming $P_1$-$P_0$ element methods of Brinkman problem~\eqref{eq:4thpurbmfem0b}-\eqref{eq:4thpurbmfem0c} and \eqref{eq:4thpurbmfemb}-\eqref{eq:4thpurbmfemc}, we can use the block-diagonal preconditioner in \cite{OlshanskiiPetersReusken2006,MardalWinther2004,CahouetChabard1988} or the approximate block-factorization preconditioner in \cite{ChenHuHuang2018}, which are robust with respect to the parameter $\varepsilon$ and mesh size $h$.
We will adopt the following approximate block-factorization preconditioner in the numerical part
\begin{align}
  \begin{pmatrix}
    A_h-B_h^T\widetilde M_h^{-1}B_h & B_h^T           \\
    B_h                             & -\widetilde M_h
  \end{pmatrix}^{-1} & =\begin{pmatrix}
    I                      & 0  \\
    \widetilde M_h^{-1}B_h & -I
  \end{pmatrix}^{-1}\begin{pmatrix}
    A_h & -B_h^T         \\
    0   & \widetilde M_h
  \end{pmatrix}^{-1} \notag       \\
                                  & =\begin{pmatrix}
    I                      & 0  \\
    \widetilde M_h^{-1}B_h & -I
  \end{pmatrix}\begin{pmatrix}
    A_h & -B_h^T         \\
    0   & \widetilde M_h
  \end{pmatrix}^{-1}, \label{eq:abfp}
\end{align}
where $\begin{pmatrix}
    A_h & B_h^T \\
    B_h & O
  \end{pmatrix}$ is the stiffness matrix of the Brinkman equation~\eqref{eq:4thpurbmfem0b}-\eqref{eq:4thpurbmfem0c}, and $\widetilde M_h=\alpha\varepsilon^{-2}M_h$ with $\alpha>0$ and $M_h$ being the mass matrix for the pressure.
Since all these solvers are based on the solvers of the Poisson equation and we use AMG method to solve the discrete methods of the Poisson equation,
the designed fast solver of the MWX element method~\eqref{mmwx0} (the MWX element method~\eqref{mmwx}) based on the discrete method \eqref{mmwx0equiv1} and \eqref{eq:4thpurbmfem0a}-\eqref{eq:4thpurbmfem0d} (the discrete method \eqref{mmwxequiv1} and \eqref{eq:4thpurbmfema}-\eqref{eq:4thpurbmfemd}) also works for the shape-regular unstructured meshes.

\section{Numerical Results}
In this section, we will provide some numerical examples to verify the theoretical convergence rates of the MWX element method~\eqref{mmwx0} and \eqref{mmwx}, and test the efficiency of a solver based on the decoupled method \eqref{eq:4thpurbmfem0a}-\eqref{eq:4thpurbmfem0d}.
Let $\Omega$ be the unit square $(0,1)^2$, and we use the uniform triangulation of $\Omega$. All the experiments are implemented with the scikit-fem library \cite{skfem2020}.

\begin{example}\label{numericalexm1}\rm 
  We first test the MWX element method~\eqref{mmwx0} with the exact solution
  \[
    u(x,y)=\sin^2(\pi x)\sin^2(\pi y).
  \]
  The right hand side $f$ is computed from \eqref{fourthorderpertub}. Notice that the solution $u$ does not have boundary layers. Take $\ell=1$.

  The energy error $\|u-u_{h0}\|_{\varepsilon, h}$ with different $\varepsilon$ and $h$ is shown in Table~\ref{ex1p1nopen}. From Table~\ref{ex1p1nopen} we observe that $\|u-u_{h0}\|_{\varepsilon, h}=O(h)$ for $\varepsilon=1, 10^{-1}, 10^{-2}$, which agrees with the theoretical convergence result \eqref{eq:error1}.
  While numerically $\|u-u_{h0}\|_{\varepsilon, h}=O(h^2)$ for $\varepsilon=10^{-3}, 10^{-4}, 10^{-5}$ in Table~\ref{ex1p1nopen}, which is superconvergent and one order higher than the theoretical convergence result \eqref{eq:error0}.

  Then examine the efficiency of solvers for the MWX element method~\eqref{mmwx0}. The stop criterion in our iterative algorithms is the relative residual is less than $10^{-8}$. And the initial guess is zero.
  First, we solve the MWX element method~\eqref{mmwx0equiv1}-\eqref{mmwx0equiv2} using the CG method with AMG method (AMG-CG) as the preconditioner. According to the iteration steps listed in the third column of Tables~\ref{tab:exm1gmreseps1}-\ref{tab:exm1gmreseps01}, equation \eqref{mmwx0equiv1} is highly efficiently solved by the AMG-CG solver.
From Table~\ref{tab:exm1pcg} we can see  that the AMG-CG solver is very efficient for $\varepsilon=10^{-4}, 10^{-5}$, and the iteration steps of the AMG-CG solver is also acceptable for $\varepsilon=10^{-2}, 10^{-3}$.

  However the AMG-CG solver deteriorates for $\varepsilon=1, 10^{-1}$, which means the AMG-CG solver doesn't work for large $\varepsilon$. To deal with this, we adopt a solver for the MWX element method~\eqref{mmwx0equiv2} based on the equivalent decoupling \eqref{eq:4thpurbmfem0a}-\eqref{eq:4thpurbmfem0d}.
  To be specific, we adopt the AMG-CG solver to solve equations \eqref{eq:4thpurbmfem0a} and \eqref{eq:4thpurbmfem0d}, and the GMRES method with the preconditioner \eqref{eq:abfp}, in which the parameter $\alpha=2$, the restart in the GMRES is $20$, and AMG is used to approximate the inverse of $A_h$. By the iteration steps listed in Tables~\ref{tab:exm1gmreseps1}-\ref{tab:exm1gmreseps01}, the AMG-CG solver is highly efficient for solving equations \eqref{eq:4thpurbmfem0a} and \eqref{eq:4thpurbmfem0d}. And the preconditioned GMRES algorithm is also very efficient and robust for the nonconforming $P_1$-$P_0$ element method of Brinkman problem~\eqref{eq:4thpurbmfem0b}-\eqref{eq:4thpurbmfem0c} for $\varepsilon=1, 10^{-1}$.

  In summary, to efficiently solve the MWX method~\eqref{mmwx0} we can employ the AMG-CG solver for small $\varepsilon$ and the GMRES method with the preconditioner \eqref{eq:abfp} for large $\varepsilon$.

  \begin{table}[htbp]
    \centering
    \caption{Error $\|u-u_{h0}\|_{\varepsilon, h}$ of the MWX method~\eqref{mmwx0} for Example~\ref{numericalexm1} with different $\varepsilon$ and $h$}
    \resizebox{\textwidth}{!}{
      \begin{tabular}{ccccccccc}
        \toprule
        \diagbox{$\varepsilon$}{$h$} & $2^{-1}$  & $2^{-2}$  & $2^{-3}$  & $2^{-4}$  & $2^{-5}$  & $2^{-6}$  & $2^{-7}$  & $2^{-8}$  \\
        \midrule
        1                            & 1.232E+01 & 7.584E+00 & 3.839E+00 & 1.896E+00 & 9.433E-01 & 4.710E-01 & 2.354E-01 & 1.177E-01 \\
                                     & -         & 0.70      & 0.98      & 1.02      & 1.01      & 1.00      & 1.00      & 1.00      \\
        $10^{-1}$                    & 1.617E+00 & 1.024E+00 & 4.386E-01 & 1.977E-01 & 9.539E-02 & 4.723E-02 & 2.356E-02 & 1.177E-02 \\
                                     & -         & 0.66      & 1.22      & 1.15      & 1.05      & 1.01      & 1.00      & 1.00      \\
        $10^{-2}$                    & 1.148E+00 & 7.291E-01 & 2.383E-01 & 6.564E-02 & 1.820E-02 & 6.062E-03 & 2.537E-03 & 1.200E-03 \\
                                     & -         & 0.66      & 1.61      & 1.86      & 1.85      & 1.59      & 1.26      & 1.08      \\
        $10^{-3}$                    & 1.143E+00 & 7.260E-01 & 2.371E-01 & 6.477E-02 & 1.665E-02 & 4.202E-03 & 1.057E-03 & 2.761E-04 \\
                                     & -         & 0.65      & 1.61      & 1.87      & 1.96      & 1.99      & 1.99      & 1.94      \\
        $10^{-4}$                    & 1.143E+00 & 7.260E-01 & 2.371E-01 & 6.477E-02 & 1.666E-02 & 4.205E-03 & 1.055E-03 & 2.641E-04 \\
                                     & -         & 0.65      & 1.61      & 1.87      & 1.96      & 1.99      & 1.99      & 2.00      \\
        $10^{-5}$                    & 1.143E+00 & 7.260E-01 & 2.371E-01 & 6.477E-02 & 1.666E-02 & 4.205E-03 & 1.055E-03 & 2.642E-04 \\
                                     & -         & 0.65      & 1.61      & 1.87      & 1.96      & 1.99      & 1.99      & 2.00      \\
        \bottomrule
      \end{tabular}%
    }
    \label{ex1p1nopen}%
  \end{table}%

  \begin{table}[htbp]
    \centering
    \caption{Iteration steps of the MWX  methods~\eqref{mmwx0} for Example~\ref{numericalexm1} with different $\varepsilon$ amd $h$}
    \begin{tabular}{ccccccccc}
      \toprule
      \diagbox{$\varepsilon$}{$h$} & $2^{-1}$ & $2^{-2}$ & $2^{-3}$ & $2^{-4}$ & $2^{-5}$ & $2^{-6}$ & $2^{-7}$ & $2^{-8}$ \\
      \midrule
      1                            & 1        & 6        & 12       & 24       & 49       & 104      & 241      & $>1000$  \\
      $10^{-1}$                    & 1        & 5        & 9        & 18       & 39       & 84       & 183      & 484      \\
      $10^{-2}$                    & 1        & 3        & 5        & 7        & 8        & 15       & 30       & 62       \\
      $10^{-3}$                    & 1        & 3        & 5        & 6        & 7        & 9        & 15       & 59       \\
      $10^{-4}$                    & 1        & 3        & 5        & 6        & 7        & 8        & 10       & 10       \\
      $10^{-5}$                    & 1        & 3        & 5        & 6        & 7        & 8        & 10       & 10       \\
      \midrule
      \#dofs                       & 25       & 81       & 289      & 1089     & 4225     & 16641    & 66049    & 263169   \\
      \bottomrule
    \end{tabular}%
    \label{tab:exm1pcg}%
  \end{table}%


  \begin{table}[htbp]
    \centering
    \caption{Iteration steps of the decoupled methods~\eqref{eq:4thpurbmfem0a}-\eqref{eq:4thpurbmfem0d} for Example~\ref{numericalexm1} with $\varepsilon=1$ and different $h$}
    \begin{tabular}{cccccc}
      \toprule
      \multirow{2}[4]{*}{$h$} & \multirow{2}[4]{*}{\#dofs} & Eq.\eqref{mmwx0equiv1} & Eq.\eqref{eq:4thpurbmfem0a} & \multicolumn{1}{l}{Eq.\eqref{eq:4thpurbmfem0b}-\eqref{eq:4thpurbmfem0c}} & Eq.\eqref{eq:4thpurbmfem0d} \\
      \cmidrule{3-6}          &                            & steps                  & steps                       & steps                                                                    & steps                       \\
      \midrule
      $2^{-1}$                & 24                         & 1                      & 1                           & 16                                                                       & 1                           \\
      $2^{-2}$                & 112                        & 1                      & 4                           & 27                                                                       & 3                           \\
      $2^{-3}$                & 480                        & 4                      & 5                           & 34                                                                       & 5                           \\
      $2^{-4}$                & 1984                       & 6                      & 7                           & 34                                                                       & 7                           \\
      $2^{-5}$                & 8064                       & 6                      & 9                           & 41                                                                       & 9                           \\
      $2^{-6}$                & 32512                      & 7                      & 11                          & 43                                                                       & 11                          \\
      $2^{-7}$                & 130560                     & 7                      & 14                          & 44                                                                       & 14                          \\
      $2^{-8}$                & 523264                     & 9                      & 17                          & 46                                                                       & 17                          \\
      $2^{-9}$                & 2095104                    & 9                      & 20                          & 50                                                                       & 21                          \\
      $2^{-10}$               & 8384512                    & 12                     & 27                          & 55                                                                       & 27                          \\
      \bottomrule
    \end{tabular}%
    \label{tab:exm1gmreseps1}%
  \end{table}%

  \begin{table}[htbp]
    \centering
    \caption{Iteration steps of the decoupled methods~\eqref{eq:4thpurbmfem0a}-\eqref{eq:4thpurbmfem0d} for Example~\ref{numericalexm1} with $\varepsilon=10^{-1}$ and different $h$}
    \begin{tabular}{cccccc}
      \toprule
      \multirow{2}[4]{*}{$h$} & \multirow{2}[4]{*}{\#dofs} & Eq.\eqref{mmwx0equiv1} & Eq.\eqref{eq:4thpurbmfem0a} & \multicolumn{1}{l}{Eq.\eqref{eq:4thpurbmfem0b}-\eqref{eq:4thpurbmfem0c}} & Eq.\eqref{eq:4thpurbmfem0d} \\
      \cmidrule{3-6}          &                            & steps                  & steps                       & steps                                                                    & steps                       \\
      \midrule
      $2^{-1}$                & 24                         & 1                      & 1                           & 26                                                                       & 1                           \\
      $2^{-2}$                & 112                        & 1                      & 3                           & 35                                                                       & 3                           \\
      $2^{-3}$                & 480                        & 4                      & 5                           & 39                                                                       & 5                           \\
      $2^{-4}$                & 1984                       & 6                      & 7                           & 50                                                                       & 7                           \\
      $2^{-5}$                & 8064                       & 6                      & 9                           & 57                                                                       & 9                           \\
      $2^{-6}$                & 32512                      & 7                      & 11                          & 74                                                                       & 11                          \\
      $2^{-7}$                & 130560                     & 7                      & 14                          & 74                                                                       & 14                          \\
      $2^{-8}$                & 523264                     & 9                      & 17                          & 78                                                                       & 17                          \\
      $2^{-9}$                & 2095104                    & 9                      & 20                          & 83                                                                       & 21                          \\
      $2^{-10}$               & 8384512                    & 12                     & 27                          & 83                                                                       & 27                          \\
      \bottomrule
    \end{tabular}%
    \label{tab:exm1gmreseps01}%
  \end{table}%

\end{example}

\begin{example}\label{numericalexm2}\rm 
  Next we verify the convergence of the MWX methods~\eqref{mmwx0} and~\eqref{mmwx} for problem \eqref{fourthorderpertub} with boundary layers.
  Let the exact solution of the Poisson equation~\eqref{poisson} be
  \[
    u^0(x,y)=\sin(\pi x)\sin(\pi y).
  \]
  Then the right hand term for both problems \eqref{fourthorderpertub} and \eqref{poisson} is set to be
  \[
    f (x, y) = -\Delta u^0 = 2\pi^2\sin(\pi x)\sin(\pi y).
  \]
  The explicit expression solution $u$ for problem \eqref{fourthorderpertub} with this right hand term is unknown.
  The solution $u$ possesses strong boundary layers when $\varepsilon$ is very small. Here we choose $\varepsilon=10^{-6}$. Errors $\|u^0-u_{h}\|_{0}$, $|u^0-u_{h}|_{1, h}$, $|u^0-u_{h}|_{2, h}$ and $\|u^0-u_{h}\|_{\varepsilon, h}$ of the discrete method~\eqref{mmwx0} for $\ell=1$ and $\ell=2$ are present in Table~\ref{ex2p1nopen} and Table~\ref{ex2p2nopen} respectively, from which we can see that $\|u^0-u_{h}\|_{0}=O(h^{1.5})$, $|u^0-u_{h}|_{1, h}=O(h^{0.5})$, $|u^0-u_{h}|_{2, h}=O(h^{-0.5})$ and $\|u^0-u_{h}\|_{\varepsilon, h}=O(h^{0.5})$. The numerical convergence rate of error $\|u^0-u_{h}\|_{\varepsilon, h}$ coincides with \eqref{eq:error3}.

  After applying the Nitsche's technique with the penalty constant $\sigma=5$, errors $\|u^0-u_{h}\|_{0}$, $|u^0-u_{h}|_{1, h}$, $\interleave u^0-u_{h}\interleave_{2, h}$ and $\interleave u^0-u_{h}\interleave_{\varepsilon, h}$ of the discrete method~\eqref{mmwx} for $\ell=1$ and $\ell=2$ are present in Table~\ref{ex2p1pen} and Table~\ref{ex2p2pen} respectively. When $\ell=1$, numerically $\|u^0-u_{h}\|_{0}=O(h^{2})$, $|u^0-u_{h}|_{1, h}=O(h^{1.5})$, $\interleave u^0-u_{h}\interleave_{2, h}=O(h^{0.5})$ and $\interleave u^0-u_{h}\interleave_{\varepsilon, h}=O(h^{1.5})$. All these convergence rates are optimal. And the convergence rates of $|u^0-u_{h}|_{1, h}$ and $\interleave u^0-u_{h}\interleave_{\varepsilon, h}$ are half order higher than the optimal rates, as indicated by \eqref{eq:error6}.
  For $\ell=2$, it is observed from Table~\ref{ex2p2pen} that  $\|u^0-u_{h}\|_{0}=O(h^{3})$, $|u^0-u_{h}|_{1, h}=O(h^{2})$, $\interleave u^0-u_{h}\interleave_{2, h}=O(h)$ and $\interleave u^0-u_{h}\interleave_{\varepsilon, h}=O(h^{2})$. Again all these convergence rates are optimal, and the convergence rate of $\interleave u^0-u_{h}\interleave_{\varepsilon, h}$ is in coincidence with \eqref{eq:error6}.

  \begin{table}[h]
    \centering
    \caption{Errors of the discrete method~\eqref{mmwx0} for Example~\ref{numericalexm2} with different $h$ when $\varepsilon=10^{-6}$ and $\ell=1$}
    \resizebox{\textwidth}{!}{
      \begin{tabular}{cccccccc}
        \toprule
        $h$                                                             & $2^{-1}$  & $2^{-2}$  & $2^{-3}$  & $2^{-4}$  & $2^{-5}$  & $2^{-6}$  & $2^{-7}$  \\
        \midrule
        \multicolumn{1}{c}{$\left\|u^0-u_{h}\right\|_{0}$}              & 2.550E-01 & 7.631E-02 & 2.576E-02 & 9.051E-03 & 3.189E-03 & 1.124E-03 & 3.968E-04 \\
                                                                        & $-$       & 1.74      & 1.57      & 1.51      & 1.51      & 1.50      & 1.50      \\
        \multicolumn{1}{c}{$\left|u^0-u_{h}\right|_{1, h}$}             & 1.871E+00 & 1.065E+00 & 6.781E-01 & 4.638E-01 & 3.240E-01 & 2.279E-01 & 1.607E-01 \\
                                                                        & $-$       & 0.81      & 0.65      & 0.55      & 0.52      & 0.51      & 0.50      \\
        \multicolumn{1}{c}{$\left|u^0-u_{h}\right|_{2, h}$}             & 1.142E+01 & 1.309E+01 & 1.807E+01 & 2.543E+01 & 3.580E+01 & 5.047E+01 & 7.124E+01 \\
                                                                        & $-$       & -0.20     & -0.46     & -0.49     & -0.49     & -0.50     & -0.50     \\
        \multicolumn{1}{c}{$\left\|u^0-u_{h}\right\|_{\varepsilon, h}$} & 1.616E+00 & 9.885E-01 & 6.523E-01 & 4.547E-01 & 3.208E-01 & 2.267E-01 & 1.603E-01 \\
                                                                        & $-$       & 0.71      & 0.60      & 0.52      & 0.50      & 0.50      & 0.50      \\
        \bottomrule
      \end{tabular}%
    }
    \label{ex2p1nopen}%
  \end{table}%

  \begin{table}[htbp]
    \centering
    \caption{Errors of the discrete method~\eqref{mmwx0} for Example~\ref{numericalexm2} with different $h$ when $\varepsilon=10^{-6}$ and $\ell=2$}
    \resizebox{\textwidth}{!}{
      \begin{tabular}{cccccccc}
        \toprule
        \multicolumn{1}{c}{$h$}                                         & $2^{-1}$  & $2^{-2}$  & $2^{-3}$  & $2^{-4}$  & $2^{-5}$  & $2^{-6}$  & $2^{-7}$  \\
        \midrule
        \multicolumn{1}{c}{$\left\|u^0-u_{h}\right\|_{0}$}              & 2.288E-01 & 7.595E-02 & 2.578E-02 & 9.000E-03 & 3.172E-03 & 1.121E-03 & 3.961E-04 \\
                                                                        & $-$       & 1.59      & 1.56      & 1.52      & 1.50      & 1.50      & 1.50      \\
        \multicolumn{1}{c}{$\left|u^0-u_{h}\right|_{1, h}$}             & 1.598E+00 & 1.002E+00 & 6.702E-01 & 4.630E-01 & 3.239E-01 & 2.279E-01 & 1.607E-01 \\
                                                                        & $-$       & 0.67      & 0.58      & 0.53      & 0.52      & 0.51      & 0.50      \\
        \multicolumn{1}{c}{$\left|u^0-u_{h}\right|_{2, h}$}             & 1.140E+01 & 1.389E+01 & 1.851E+01 & 2.561E+01 & 3.587E+01 & 5.049E+01 & 7.125E+01 \\
                                                                        & $-$       & -0.28     & -0.41     & -0.47     & -0.49     & -0.49     & -0.50     \\
        \multicolumn{1}{c}{$\left\|u^0-u_{h}\right\|_{\varepsilon, h}$} & 1.369E+00 & 9.258E-01 & 6.444E-01 & 4.540E-01 & 3.207E-01 & 2.267E-01 & 1.603E-01 \\
                                                                        & $-$       & 0.56      & 0.52      & 0.51      & 0.50      & 0.50      & 0.50      \\
        \bottomrule
      \end{tabular}%
    }
    \label{ex2p2nopen}%
  \end{table}%

  \begin{table}[htbp]
    \centering
    \caption{Errors of the discrete method~\eqref{mmwx} for Example~\ref{numericalexm2} with different $h$ when $\varepsilon=10^{-6}$ and $\ell=1$}
    \resizebox{\textwidth}{!}{
      \begin{tabular}{cccccccc}
        \toprule
        $h$                                                                   & $2^{-1}$  & $2^{-2}$  & $2^{-3}$  & $2^{-4}$  & $2^{-5}$  & $2^{-6}$  & $2^{-7}$  \\
        \midrule
        \multicolumn{1}{c}{$\left\|u^{0}-u_{h}\right\|_{0}$}                      & 2.234E-01 & 4.892E-02 & 9.360E-03 & 1.948E-03 & 4.436E-04 & 1.062E-04 & 2.601E-05 \\
                                                                              & -         & 2.19      & 2.39      & 2.26      & 2.13      & 2.06      & 2.03      \\
        \multicolumn{1}{c}{$\left|u^{0}-u_{h}\right|_{1, h}$}                     & 1.539E+00 & 5.597E-01 & 1.759E-01 & 5.513E-02 & 1.787E-02 & 5.983E-03 & 2.050E-03 \\
                                                                              & -         & 1.46      & 1.67      & 1.67      & 1.63      & 1.58      & 1.55      \\
        \multicolumn{1}{c}{$\interleave u^{0}-u_{h}\interleave_{2, h}$}           & 9.695E+00 & 6.714E+00 & 4.320E+00 & 2.801E+00 & 1.870E+00 & 1.279E+00 & 8.888E-01 \\
                                                                              & -         & 0.53      & 0.64      & 0.62      & 0.58      & 0.55      & 0.53      \\
        \multicolumn{1}{c}{$\interleave u^{0}-u_{h}\interleave_{\varepsilon, h}$} & 1.316E+00 & 5.108E-01 & 1.666E-01 & 5.318E-02 & 1.742E-02 & 5.877E-03 & 2.024E-03 \\
                                                                              & -         & 1.37      & 1.62      & 1.65      & 1.61      & 1.57      & 1.54      \\
        \bottomrule
      \end{tabular}%
    }
    \label{ex2p1pen}%
  \end{table}%

  \begin{table}[htbp]
    \centering
    \caption{Errors of the discrete method~\eqref{mmwx} for Example~\ref{numericalexm2} with different $h$ when $\varepsilon=10^{-6}$ and $\ell=2$}
    \resizebox{\textwidth}{!}{
      \begin{tabular}{cccccccc}
        \toprule
        \multicolumn{1}{c}{$h$}                                               & $2^{-1}$  & $2^{-2}$  & $2^{-3}$  & $2^{-4}$  & $2^{-5}$  & $2^{-6}$  & $2^{-7}$  \\
        \midrule
        \multicolumn{1}{c}{$\left\|u^{0}-u_{h}\right\|_{0}$}                      & 6.445E-02 & 7.345E-03 & 7.175E-04 & 7.567E-05 & 8.842E-06 & 1.083E-06 & 1.347E-07 \\
                                                                              & -         & 3.13      & 3.36      & 3.25      & 3.10      & 3.03      & 3.01      \\
        \multicolumn{1}{c}{$\left|u^{0}-u_{h}\right|_{1, h}$}                     & 4.834E-01 & 1.293E-01 & 3.343E-02 & 8.446E-03 & 2.115E-03 & 5.286E-04 & 1.321E-04 \\
                                                                              & -         & 1.90      & 1.95      & 1.98      & 2.00      & 2.00      & 2.00      \\
        \multicolumn{1}{c}{$\interleave u^{0}-u_{h}\interleave_{2, h}$}           & 6.369E+00 & 3.434E+00 & 1.765E+00 & 8.873E-01 & 4.434E-01 & 2.214E-01 & 1.106E-01 \\
                                                                              & -         & 0.89      & 0.96      & 0.99      & 1.00      & 1.00      & 1.00      \\
        \multicolumn{1}{c}{$\interleave u^{0}-u_{h}\interleave_{\varepsilon, h}$} & 4.190E-01 & 1.220E-01 & 3.271E-02 & 8.370E-03 & 2.106E-03 & 5.275E-04 & 1.319E-04 \\
                                                                              & -         & 1.78      & 1.90      & 1.97      & 1.99      & 2.00      & 2.00      \\
        \bottomrule
      \end{tabular}%
    }
    \label{ex2p2pen}
  \end{table}%
\end{example}


\end{document}